\documentclass[12pt]{amsart}

\usepackage[T1]{fontenc}
\usepackage[utf8]{inputenc}
\usepackage{amsmath,amsthm,amssymb,amsfonts}
\usepackage{enumerate}
\usepackage{mathrsfs}
\usepackage{hyperref}

\allowdisplaybreaks

\hypersetup{
  colorlinks=true,
  linkcolor=blue,
  citecolor=blue,
  urlcolor=blue
}

\newcommand{\Z}{\mathbb Z}
\newcommand{\N}{\mathbb N}
\newcommand{\Zp}{\mathbb Z_p}
\newcommand{\Cp}{\mathbb C_p}
\newcommand{\Fp}{\mathbb F_p}
\newcommand{\Fq}{\mathbb F_q}
\newcommand{\Cinf}{\mathbb C_\infty}
\newcommand{\Kinf}{K_\infty}
\newcommand{\cO}{\mathcal O}
\newcommand{\cA}{\mathcal A}
\newcommand{\cL}{\mathbf L}
\newcommand{\cH}{\mathbf H}
\newcommand{\cB}{\mathcal B}
\newcommand{\cM}{\mathcal M}
\newcommand{\fm}{\mathfrak m}
\newcommand{\CD}{\mathfrak D}
\newcommand{\dd}{\,\mathrm d}
\newcommand{\norm}[1]{\lVert#1\rVert}
\newcommand{\rbinom}[2]{\overline{\binom{#1}{#2}}}
\newcommand{\dirlim}{\varinjlim}
\newcommand{\invlim}{\varprojlim}

\DeclareMathOperator{\Tr}{Tr}
\DeclareMathOperator{\Spec}{Spec}
\DeclareMathOperator{\Def}{Def_q}

\newtheorem{theorem}{Theorem}[section]
\newtheorem{lemma}[theorem]{Lemma}
\newtheorem{proposition}[theorem]{Proposition}
\newtheorem{corollary}[theorem]{Corollary}

\theoremstyle{definition}
\newtheorem{definition}[theorem]{Definition}
\newtheorem{remark}[theorem]{Remark}
\newtheorem{example}[theorem]{Example}

\numberwithin{equation}{section}
\begin{document}

\title[Spectral zeta functions in positive characteristic]
{Spectral zeta functions in positive characteristic}

\author{Su Hu}
\address{Department of Mathematics, South China University of Technology,
Guangzhou, Guangdong 510640, China}
\email{mahusu@scut.edu.cn}

\author{Min-Soo Kim}
\address{Department of Mathematics Education, Kyungnam University,
Changwon, Gyeongnam 51767, Republic of Korea}
\email{mskim@kyungnam.ac.kr}


\begin{abstract}
Spectral zeta functions and zeta-regularized determinants encode
the spectrum of an operator in analytic functions.
Over the completion $\mathbb C_\infty$ of an algebraic closure of
$\mathbb F_q((\pi))$, two features of characteristic~$p$ break
this classical picture. Complex powers have exponents that vary
only continuously in $\mathbb Z_p$, not analytically. Scalar
traces record eigenvalue multiplicities only modulo~$p$.
We develop a resolvent-based theory that treats these two
obstructions separately.

For the first, inverse Stieltjes duality turns suitable scalar
resolvents into locally analytic spectral distributions on
$\mathbb F_q[[\pi]]$. A single exponent of the Mellin transform
already determines the distribution and the entire Mellin family.
We characterize exactly which power series arise in this way, by
a growth condition on Carlitz moments. The criterion depends on
the support: bounded Taylor coefficients are not sufficient.
For stationary digit measures, the transform is the unique fixed
point of an explicit contraction. The first $q-1$ nonconstant
moments recover the measure.

For the second, we replace traces by the Frobenius defect
$F/\sigma(F)$, where $\sigma$ raises coefficients to their
$q$th powers. On power series with small coefficients, this map
is a group automorphism that preserves Gauss distances, and a
convergent infinite product inverts it. Polynomial defects
produce entire functions. They characterize exactly the diagonal
spectra that are finite unions of forward Frobenius orbits,
counted with multiplicity. For the shifted Carlitz operator,
the resulting Fredholm determinant is a normalized reciprocal
of the Anderson--Thakur function, and its Newton polygon is
determined exactly.
\end{abstract}
\maketitle

%======================================================================
\section{Introduction}
\label{sec:intro}
%======================================================================

Spectral zeta functions provide a way to study an operator through
analytic functions associated with its spectrum.
For a positive discrete spectrum
$\{\lambda_n\}_{n\ge0}$, with eigenvalues repeated according to
multiplicity, the classical spectral zeta function is initially
defined by
\[
Z(s,\lambda)=\sum_{n=0}^{\infty}\frac{1}{(\lambda_n+\lambda)^{s}}
\]
in a region of convergence.
If it admits a meromorphic continuation that is regular at $s=0$,
one defines the associated zeta-regularized determinant by
\[
\log D(\lambda)
=
-\left.\frac{\partial}{\partial s}Z(s,\lambda)\right|_{s=0}.
\]
The work of Hawking and Voros established the importance of this
construction in mathematical physics and spectral analysis
\cite{Hawking,Voros}.
Among its applications are the evaluation of functional determinants
and the study of quantum partition functions; see \cite{Elizalde}.

The discrete sum is only one representation of spectral information.
When an appropriate scalar spectral measure is available, a zeta
function can instead be expressed as the integral of a power kernel
against that measure.
The resolvent provides another route to the same information:
in familiar spectral settings, suitable scalarizations of
resolvents are Stieltjes transforms of spectral measures.
This point of view is particularly useful when the spectrum is not
discrete or when the ordinary operator trace is unavailable.
It also suggests a non-archimedean construction, since the resolvent
and the Stieltjes kernel make sense without a real heat-time
parameter.

In \cite{HuKim}, we introduced $p$-adic spectral zeta functions
by applying $p$-adic Mellin transforms to locally analytic
interpolations of discrete spectral data.
That construction is modeled on the $p$-adic Hurwitz and
Hurwitz-type Euler zeta functions.
It uses prescribed distributions on the interpolation parameter
space $\Zp$, and therefore requires the spectral data to fit a
suitable interpolation scheme.
In \cite{HuKimInverse}, we developed a resolvent-based approach:
a chosen scalarization of a bounded operator's resolvent is
converted into a generalized spectral distribution by the
inverse Stieltjes transform.
The distribution is then evaluated on a zeta kernel.
This replaces the interpolation parameter by the spectral
distribution itself and accommodates suitable non-discrete
compact spectra.

The present paper asks how this approach changes over a local
field of positive characteristic.
Throughout, we work over
\[
\Kinf=\Fq((\pi)),\qquad
\cO=\Fq[[\pi]],\qquad
\Cinf=\widehat{\overline{\Kinf}},
\qquad q=p^e,
\]
where $p$ is prime and $e\ge1$, and we normalize the absolute value
by $|\pi|=q^{-1}$.
Our functions, distributions, and scalarized operator expressions
take values in $\Cinf$.
Thus we study positive-characteristic-valued analysis, rather than
complex-valued harmonic analysis on a non-archimedean space.

Two differences from the characteristic-zero theory are decisive.
First, a principal unit in $\Cinf$ admits continuous $\Zp$-powers,
defined by reducing binomial coefficients modulo $p$.
This does not make the exponent a $\Cinf$-analytic variable.
Consequently, the usual prescription involving differentiation
with respect to the exponent at zero does not transfer directly.
Second, a scalar trace records an eigenvalue multiplicity only
modulo $p$.
Even complete knowledge of all scalar power traces therefore
need not determine the characteristic polynomial or a Fredholm
determinant.
The analytic reconstruction of a scalar distribution and the
reconstruction of a determinant must be treated as distinct
problems.

Our first problem is to determine which analytic functions arise
from spectral distributions.
Let $\cL^*(\cO)$ denote the continuous $\Cinf$-linear dual of the
space of locally analytic functions on $\cO$.
For $\mu\in\cL^*(\cO)$, we consider
\begin{equation}\label{eq:intro-mellin}
\cM_\mu(y,t)=\mu_x\bigl((1+tx)^y\bigr),
\qquad y\in\Zp,\quad |t|<1.
\end{equation}
The parameter $t$ is analytic, while $y$ is a continuous exponent
parameter.
The associated relative spectral zeta transform is
\begin{equation}\label{eq:intro-zeta}
\zeta_\infty^\mu(y,\lambda)
=
\cM_\mu(-y,\lambda^{-1}),
\qquad |\lambda|>1.
\end{equation}
The word ``relative'' refers to the use of the normalized kernel
$(1+x/\lambda)^{-y}$, without a leading factor depending only
on $\lambda$.
It does not refer to a subtraction of two operator spectra.
The relation with the exponent and normalization conventions
of \cite{HuKimInverse} is explained in
Remark~\ref{rem:padic-zeta-normalization}.

To give \eqref{eq:intro-mellin} a spectral interpretation,
let $T$ be a bounded operator on a $\Cinf$-Banach space and let
$\Lambda$ be a continuous linear functional on its bounded
operator algebra.
Under the admissibility conditions specified in
Section~\ref{sec:distributions}, the scalar resolvent
\[
F_{T,\Lambda}(z)=\Lambda((zI-T)^{-1})
\]
is the Stieltjes transform of a unique distribution
$\mu_{T,\Lambda}$ on $\cO$.
Its ordinary moments satisfy
\[
\mu_{T,\Lambda}(x^k)=\Lambda(T^k).
\]
No cyclicity is assumed for $\Lambda$.
For example, $\Lambda(B)=\ell(Bu)$ is a scalar matrix coefficient.
The distribution is therefore attached to the pair
$(T,\Lambda)$, not to the spectrum of $T$ alone.

The exponent $y=-1$ turns the Mellin kernel into a resolvent kernel.
Writing
\[
H_\mu(t)=\cM_\mu(-1,t),
\]
one has
\[
S_\mu(z)=z^{-1}H_\mu(-z^{-1}),\qquad |z|>1.
\]
Thus this single exponent contains all ordinary moments.
The substantive range question is which Taylor coefficient
sequences come from a locally analytic distribution on the
specified compact set $\cO$.

We answer this question in Carlitz coordinates.
Let
\[
G_n(X)=\sum_{k=0}^n g_{n,k}X^k
\]
be the general Carlitz polynomials associated with $\pi$.
Theorem~\ref{thm:mellin-reconstruction} shows that a formal series
$H(t)=\sum_{k\ge0}a_kt^k$ is the $y=-1$ Mellin slice of a unique
distribution on $\cO$ if and only if
\begin{equation}\label{eq:intro-range}
\limsup_{n\to\infty}
\left|
\sum_{k=0}^n g_{n,k}(-1)^ka_k
\right|^{1/n}
\le1.
\end{equation}
This condition automatically implies boundedness of the
coefficients $a_k$ and convergence on $|t|<1$.
It also determines the full Mellin family by its binomial expansion.
The criterion is sensitive to the support set:
bounded Taylor coefficients alone are not sufficient.
Its proof is a dual application of weighted Carlitz interpolation.
For the precise Banach norms used here, we include a direct proof
of the required orthonormal basis statement.

We next consider stationary digit measures.
Their self-similarity gives a weighted composition equation for
\eqref{eq:intro-mellin}.
On every closed disc of radius less than one, the corresponding
operator is a strict contraction on functions with a prescribed
constant term.
The transform is therefore the unique normalized analytic fixed
point.
The iterates give finite digit sums with a uniform error bound.
Within this class of measures, the moments
$m_1,\ldots,m_{q-1}$ determine all digit weights; the zeroth
moment is already fixed to be one.
The position operator on $C(\cO,\Cinf)$ realizes these measures
as scalar spectral distributions and provides a non-discrete
spectral model.

Our second problem concerns determinants.
Since scalar traces are insufficient in general, we prescribe
additional data through an equation involving coefficient Frobenius.
Let
\[
\mathcal U=
\left\{
1+\sum_{r\ge1}a_rZ^r:\ |a_r|<1
\right\},
\]
and let $\sigma$ fix $Z$ and raise coefficients to their $q$th powers.
We call
\[
\Def(F)=\frac{F}{\sigma(F)}
\]
the Frobenius defect of $F$.
For independently prescribed $\Phi\in\mathcal U$, we study
\[
F=\Phi\,\sigma(F),\qquad F\in\mathcal U.
\]
The small-coefficient condition removes the ambiguity caused
by nontrivial series fixed by $\sigma$.

Theorem~\ref{thm:frobenius-resolution} proves that $\Def$ is a
group automorphism of $\mathcal U$.
Its inverse is constructed both by an Artin--Schreier recursion
and by the product
\[
\Def^{-1}(\Phi)=\prod_{j\ge0}\sigma^j(\Phi).
\]
The product converges coefficientwise and in the Gauss norm on
every closed disc of radius less than one.
Moreover, the defect map preserves the $Z$-adic filtration and
the Gauss distances on these discs.

A polynomial defect imposes a more restrictive spectral structure.
If
\[
\Phi(Z)=\prod_{i=1}^s(1-c_iZ),
\qquad 0<|c_i|<1,
\]
then its normalized solution is the entire function
\[
F(Z)=\prod_{i=1}^s\prod_{n\ge0}(1-c_i^{q^n}Z).
\]
Consequently, a diagonal determinant with small eigenvalues
tending to zero has polynomial Frobenius defect if and only if
its nonzero eigenvalue multiset is a finite multiset union of
complete forward Frobenius orbits.
Multiplicities are part of the statement.
Zero eigenvalues are excluded because they do not affect
$\det(I-ZA)$.
The polynomial condition is essential:
every $F\in\mathcal U$ has its own, possibly nonpolynomial, defect.

The shifted Carlitz--Bargmann--Fock operator
\[
A=\tau d+\pi I
\]
provides the single-orbit example.
Its eigenvalues are $\pi^{q^n}$, and its Fredholm determinant is
\[
\Delta_A(Z)=\prod_{n\ge0}(1-\pi^{q^n}Z).
\]
The diagonal relation supplies the defect $1-\pi Z$
independently of scalar traces.
We determine all coefficient valuations and successive Newton slopes.
If $\theta=\pi^{-1}$ and $\omega_\theta$ denotes the
Anderson--Thakur function, its standard product formula gives
\[
\Delta_A(Z)=\frac{\omega_\theta(0)}{\omega_\theta(Z)}.
\]
Thus this operator realizes a normalized reciprocal of an
established special function as a Fredholm determinant
\cite{AndersonThakur,AnglesPellarin}.
The identity is a comparison of the two product formulas,
not a new definition of the Anderson--Thakur function.

The two reconstruction problems meet at the logarithmic
derivative of a diagonal determinant.
When its eigenvalues $\lambda_n$ belong to $\cO$, satisfy
$|\lambda_n|<1$, and tend to zero, the bounded measure
\[
\nu_\lambda=\sum_{n\ge0}\lambda_n\delta_{\lambda_n}
\]
satisfies
\[
\frac{\Delta_\lambda'(Z)}{\Delta_\lambda(Z)}
=
-\cM_{\nu_\lambda}(-1,-Z),\qquad |Z|<1.
\]
The Mellin transform describes the logarithmic derivative,
whereas prescribed Frobenius data determine the normalized
determinant.
This also explains why the determinant terminology differs
from the normalized $p$-adic log-gamma functionals of
\cite{HuKimInverse}; see
Remark~\ref{rem:padic-logarithmic-comparison}.

The analytic tools used here are established ones:
the Shnirelman--Vishik contour method
\cite{Koblitz,Vishik}, Carlitz polynomial systems and the digit
principle \cite{Carlitz,Wagner,Conrad}, and weighted interpolation
theory \cite{Amice,Kochubei}.
Goss's formal Mellin transform gives a related description
of bounded measures in divided-power coordinates
\cite{GossFormal}.
Our emphasis is on the explicit support-sensitive range criterion,
the analytic fixed-point description of digit transforms,
and normalized Frobenius reconstruction with its
polynomial-defect classification.

The local Carlitz interpolation system uses $\pi$ and is not
identified with the global Carlitz module attached to
$\theta=\pi^{-1}$.
No identification of our spectral zeta transforms with
Goss or Taelman special values is asserted.

Section~\ref{sec:distributions} develops inverse Stieltjes
duality, Carlitz coordinates, and scalar spectral distributions.
Section~\ref{sec:mellin} establishes the Mellin calculus
and reconstruction theorem.
Section~\ref{sec:digits} treats digit measures.
Section~\ref{sec:fredholm} develops Frobenius-normalized
determinants and applies the theory to the shifted Carlitz operator.

%======================================================================
\section{Spectral distributions and Carlitz coordinates}
\label{sec:distributions}
%======================================================================

There are two complementary descriptions of a locally analytic
distribution: an exterior Cauchy transform and a sequence of
Carlitz moments.
We develop both descriptions with explicit Banach norms.

Write $\N_0=\{0,1,2,\ldots\}$ and
\[
\fm=\{z\in\Cinf:|z|<1\}.
\]
For $a\ne0$, put
\[
v_\pi(a)=-\log_q|a|,
\qquad v_\pi(0)=+\infty.
\]
We use
\[
D(a,r)=\{z\in\Cinf:|z-a|\le r\},
\qquad
D(a,r^-)=\{z\in\Cinf:|z-a|<r\}.
\]
An analytic function on an open disc means a power series
converging on every strictly smaller closed disc.
This is stronger than being locally analytic at every point.

\subsection{Analytic neighborhoods and the Cauchy transform}

The compact set $\cO$ is not the closed unit disc of $\Cinf$.
For $n\ge0$, let
\[
\cA_n=
\left\{
\sum_{j=0}^{n-1}a_j\pi^j:\ a_j\in\Fq
\right\},
\qquad \cA_0=\{0\}.
\]
Choose $\gamma_n\in\Cinf$ with
\[
|\gamma_n|=r_n=q^{-n+1/2},
\]
and set
\[
U_n=\bigcup_{a\in\cA_n}D(a,r_n),
\qquad
\Omega_n=
\mathbb P^1(\Cinf)\setminus
\bigcup_{a\in\cA_n}D(a,r_n^-).
\]
The discs at each level are disjoint, since distinct
$a,b\in\cA_n$ satisfy
\[
|a-b|\ge q^{-(n-1)}>r_n.
\]
Every level-$(n+1)$ disc lies strictly inside its parent:
the centers differ by at most $q^{-n}$, while
\[
r_{n+1}<q^{-n}<r_n.
\]
Furthermore,
\[
\bigcap_nU_n=\cO,
\qquad
\bigcup_n\Omega_n=\mathbb P^1(\Cinf)\setminus\cO.
\]
Indeed, a point in every $U_n$ is a limit of representatives
in $\cA_n$ and hence belongs to $\cO$.
Conversely, every point outside the compact set $\cO$ has
positive distance from it.

Define
\[
B_n=
\prod_{a\in\cA_n}
\Cinf\left\langle\frac{X-a}{\gamma_n}\right\rangle.
\]
For
\[
f|_{D(a,r_n)}
=
\sum_{k\ge0}f_{a,k}(X-a)^k,
\]
we have
\[
|f_{a,k}|r_n^k\longrightarrow0,
\qquad
\norm{f}_n=\max_{a,k}|f_{a,k}|r_n^k.
\]
Restriction gives continuous injections $B_n\to B_{n+1}$.
Injectivity follows from the identity principle on each
parent disc.

\begin{definition}\label{def:distribution}
The space of locally analytic functions on $\cO$ is
\[
\cL(\cO)=\dirlim_n B_n,
\]
equipped with its locally convex inductive-limit topology.
Its continuous $\Cinf$-linear dual is denoted by
$\cL^*(\cO)$ and is equipped with the seminorms
\[
\norm{\mu}_n
=
\sup_{\norm{f}_n\le1}|\mu(f)|.
\]
We write
\[
\mu(f)=\int_{\cO}f(x)\dd\mu(x).
\]
\end{definition}

Compactness identifies $\cL(\cO)$ with analytic germs on
neighborhoods of $\cO$.
The identity principle identifies these germs with their
restrictions to $\cO$.
A bounded measure is a continuous linear functional on
$C(\cO,\Cinf)$.
Every bounded measure restricts to an element of $\cL^*(\cO)$.
All dual-topology statements refer to the displayed seminorms.

\begin{definition}[Characteristic $p$ Shnirelman integral]
 For $|\gamma|=r>0$, define
\[
\mathscr A_{a,\gamma}(h)
=
\lim_{\substack{N\to\infty\\p\nmid N}}
\frac1N\sum_{\eta^N=1}h(a+\gamma\eta),
\]
whenever the limit exists, and put
\[
\mathscr C_{a,\gamma}(h)
=
\mathscr A_{a,\gamma}((z-a)h(z)).
\] \end{definition}

\begin{lemma}[Characteristic $p$ Cauchy integral formula]
\label{lem:contours}
Let $r\in|\Cinf^\times|$ and choose $|\gamma|=r$.
Suppose
\[
h(z)=\sum_{k\in\Z}c_k(z-a)^k,
\qquad
|c_k|r^k\longrightarrow0
\quad\text{as }|k|\to\infty.
\]
Then the Laurent series converges uniformly on $|z-a|=r$,
and
\[
\mathscr A_{a,\gamma}(h)=c_0,
\qquad
\mathscr C_{a,\gamma}(h)=c_{-1}.
\]
Moreover,
\[
|\mathscr C_{a,\gamma}(h)|
\le r\sup_{|z-a|=r}|h(z)|.
\]
If $f$ is Tate analytic on $D(a,r)$ and $|b-a|<r$, then
\[
\mathscr C_{a,\gamma}
\left(\frac{f(z)}{(z-b)^{m+1}}\right)
=
D^{(m)}f(b),
\qquad m\ge0,
\]
where $D^{(m)}f(b)$ is the coefficient of $U^m$ in $f(b+U)$.
\end{lemma}

\begin{proof}
For $p\nmid N$, the $N$th roots of unity are distinct,
$|N|=1$, and
\[
\frac1N\sum_{\eta^N=1}\eta^k
=
\begin{cases}
1,&N\mid k,\\
0,&N\nmid k.
\end{cases}
\]
For fixed $k\ne0$, divisibility is impossible when $N>|k|$.
This proves coefficient extraction for Laurent polynomials.
Each finite averaging operator has norm at most one, so
uniform approximation proves the assertion for the stated
Laurent series.
Multiplication by $z-a$ gives the contour estimate.

Changing the center of a Tate series from $a$ to $b$
preserves convergence on the closed disc.
If $f(z)=\sum_{n\ge0}a_n(z-a)^n$, its coefficient $c_k$
at the new center satisfies
\[
c_k=\sum_{n\ge k}\binom nk a_n(b-a)^{n-k},
\qquad
|c_k|r^k
\le
\sup_{n\ge k}|a_n|r^n\longrightarrow0.
\]
Thus
\[
f(z)=\sum_{k\ge0}D^{(k)}f(b)(z-b)^k.
\]
For $j\ge1$,
\[
(z-b)^{-j}
=
\sum_{\ell\ge0}
\binom{j+\ell-1}{\ell}
(b-a)^\ell(z-a)^{-j-\ell}
\]
converges uniformly on the circle.
Its coefficient of $(z-a)^{-1}$ is one for $j=1$
and zero otherwise.
After dividing the Taylor expansion of $f$ by
$(z-b)^{m+1}$, only the term $k=m$ contributes.
\end{proof}

Let $H_n$ be the supremum-norm completion on $\Omega_n$
of finite principal-part sums
\[
\sum_{a\in\cA_n}\sum_{k=0}^{N_a}
c_{a,k}(z-a)^{-k-1}.
\]
Then
\begin{equation}\label{eq:exterior-norm}
H_n=
\left\{
\sum_{a\in\cA_n}\sum_{k\ge0}c_{a,k}(z-a)^{-k-1}:
|c_{a,k}|r_n^{-k-1}\to0
\right\},
\end{equation}
and
\[
\norm{F}_{H_n}
=
\max_{a,k}|c_{a,k}|r_n^{-k-1}.
\]
The upper estimate is immediate.
For the reverse estimate, use
\[
c_{a,k}
=
\mathscr A_{a,\gamma_n}
\bigl((z-a)^{k+1}F(z)\bigr).
\]
Principal parts at other centers are analytic on
$D(a,r_n)$ and contribute zero.

Restriction gives continuous maps $H_{n+1}\to H_n$.
A child pole expands around its parent with ratio at most
$q^{-1/2}$.
Put
\[
\cH_0(\cO)=\invlim_nH_n.
\]
Its elements are compatible exterior analytic functions
vanishing at infinity.

\begin{proposition}[Characteristic  $p$ inverse Stieltjes duality]
\label{prop:duality}
The maps
\[
S_\mu(z)=\mu_x((z-x)^{-1})
\]
and
\[
(VF)(f)
=
\sum_{a\in\cA_n}
\mathscr C_{a,\gamma_n}(F_n(z)f(z)),
\qquad f\in B_n,
\]
are mutually inverse continuous linear maps
\[
\cL^*(\cO)
\ \underset{V}{\stackrel{S}{\rightleftarrows}}\
\cH_0(\cO).
\]
They satisfy
\[
\norm{S_\mu}_{H_n}\le r_n^{-1}\norm{\mu}_{n+1},
\qquad
\norm{VF}_n\le r_n\norm{F_n}_{H_n}.
\]
\end{proposition}

\begin{proof}
Let $\mathbf1_{a,n}$ be the indicator of $a+\pi^n\cO$,
represented as a piecewise constant function on $U_n$, and put
\[
m_{a,k}^{(n)}
=
\mu\bigl(\mathbf1_{a,n}(x)(x-a)^k\bigr).
\]
On the level-$(n+1)$ child discs of $a$, $|x-a|\le q^{-n}$.
Hence
\[
|m_{a,k}^{(n)}|
\le\norm{\mu}_{n+1}q^{-nk}.
\]
Expanding the Cauchy kernel in $B_{n+1}$ gives
\[
S_\mu(z)
=
\sum_{a\in\cA_n}\sum_{k\ge0}
m_{a,k}^{(n)}(z-a)^{-k-1},
\qquad z\in\Omega_n.
\]
Its weighted coefficients satisfy
\[
|m_{a,k}^{(n)}|r_n^{-k-1}
\le
r_n^{-1}\norm{\mu}_{n+1}q^{-k/2}.
\]
This proves membership in $H_n$ and the first estimate.

The contour estimate proves the bound for $V$.
For level compatibility, first take a finite principal-part
sum at level $n+1$.
The parent contour is the sum of its child-pole contributions
by Lemma~\ref{lem:contours}.
The child contours give the same value.
Uniform approximation in $H_{n+1}$ extends the identity
to every compatible $F$.

If $f\in B_n$, coefficient extraction gives
\[
(VS_\mu)(f)
=
\sum_{a,k}f_{a,k}m_{a,k}^{(n)}
=
\mu(f),
\]
where continuity on $B_n$ justifies the last equality.
Conversely, testing $VF$ on
$\mathbf1_{a,n}(x)(x-a)^k$ recovers the coefficient
$c_{a,k}^{(n)}$ of $F_n$ in \eqref{eq:exterior-norm}.
The local Cauchy expansion therefore gives $SVF=F$.
The estimates prove continuity in both directions.
\end{proof}

\begin{remark}[Comparison with the $p$-adic inverse Stieltjes transform]
\label{rem:padic-stieltjes}
Proposition~\ref{prop:duality} uses the same
Shnirelman--Vishik mechanism as \cite[pp.~138--142, Theorem]{Koblitz}
or \cite[Theorem~2.17]{HuKimInverse}; see also the original works
\cite{Shnirelman,Vishik}.
In both characteristics, the restriction $p\nmid N$
ensures $|N|=1$ and hence a norm bound of one for the
averaging operators.
In characteristic $p$, it also ensures that $N$ is nonzero
in the coefficient field and that $X^N-1$ is separable.
In characteristic zero, the roots of $X^N-1$ are distinct
for every positive integer $N$.

The higher-order Cauchy formula illustrates the change
in differential notation.
Over $\Cp$, the corresponding Taylor coefficient is
$f^{(m)}(b)/m!$; here it is the Hasse derivative
$D^{(m)}f(b)$.
The Hasse formulation remains meaningful in characteristic
$p$ without division by $m!$.

In either setting, termwise evaluation by a locally analytic
distribution requires convergence in a fixed analytic Banach
space, not merely uniform convergence on the compact support.
The strict separation of radii in the present disc model
provides that analytic-norm convergence.
\end{remark}

\subsection{The Carlitz coordinate criterion}

Set
\[
[j]_\pi=\pi^{q^j}-\pi,\qquad
\CD_0=1,\qquad
\CD_j=[j]_\pi\CD_{j-1}^q
\quad(j\ge1),
\]
and define
\[
e_j(X)=\prod_{a\in\cA_j}(X-a),
\qquad
f_j(X)=\frac{e_j(X)}{\CD_j}.
\]
In particular, $e_0(X)=f_0(X)=X$.
The polynomials $e_j$ are $\Fq$-linear, since $\cA_j$
is a finite $\Fq$-vector space.
For
\[
n=\sum_{j\ge0}n_jq^j,\qquad 0\le n_j<q,
\]
put
\[
G_n(X)=\prod_{j\ge0}f_j(X)^{n_j}.
\]
Then $G_0=1$ and $\deg G_n=n$.

The continuous Carlitz basis goes back to
\cite{Carlitz,Wagner}; its relation with the digit principle
is explained in \cite{Conrad}.
Weighted versions occur in locally analytic interpolation
theory \cite{Amice, Kochubei}.
We record a direct proof for the precise norms used below,
so that the ensuing norm identity does not depend on an
unstated equivalence of Banach norms.

For $h\ge0$, define
\[
\mathscr E_h
=
\prod_{a\in\cA_h}
\Cinf\left\langle\frac{X-a}{\pi^h}\right\rangle,
\qquad
\eta_{n,h}
=
\sum_{i>h}\left\lfloor\frac n{q^i}\right\rfloor.
\]

\begin{lemma}[Weighted Carlitz bases]
\label{lem:weighted-carlitz-basis}
The polynomials
\[
\pi^{\eta_{n,h}}G_n,\qquad n\ge0,
\]
form an orthonormal basis of $\mathscr E_h$.
Also, $(G_n)_{n\ge0}$ is an orthonormal basis of
$C(\cO,\Cinf)$.
In particular, if $f=\sum_{n\ge0}a_nG_n$, then
\[
f\in\mathscr E_h
\quad\Longleftrightarrow\quad
|a_n|q^{\eta_{n,h}}\longrightarrow0,
\]
and in that case
\[
\norm{f}_{\mathscr E_h}
=
\sup_n |a_n|q^{\eta_{n,h}}.
\]
\end{lemma}

\begin{proof}
We first work over $\Kinf$.
Write
\[
w_j=\frac{q^j-1}{q-1}.
\]
Then $v_\pi(\CD_j)=w_j$.
A direct count gives
\[
\sum_{b\in\cA_j\setminus\{0\}}v_\pi(b)=w_j-j.
\]
Consequently, for $i\ge j$,
\[
v_\pi(f_j(\pi^i))=i-j,
\]
whereas $f_j(\pi^i)=0$ for $i<j$.
Thus, if
$x=\sum_{i\ge0}x_i\pi^i\in\cO$, additivity and continuity give
\[
\overline{f_j(x)}=u_jx_j,
\qquad
u_j=\overline{f_j(\pi^j)}\in\Fq^\times.
\]
Here reduction is modulo $\pi$.

Fix $h$ and write $X=a+\pi^hY$ on the component indexed
by $a\in\cA_h$.
For $j<h$,
\[
f_j(a+\pi^hY)=f_j(a)+f_j(\pi^hY),
\qquad
\norm{f_j(\pi^hY)}=q^{j-h}<1.
\]
Hence its reduction is the constant $\overline{f_j(a)}$.

For $j\ge h$, $e_j(a)=0$ and therefore
\[
e_j(a+\pi^hY)=e_j(\pi^hY).
\]
Using multiplicativity of the Gauss norm,
\[
v_\pi\bigl(\norm{e_j(\pi^hY)}\bigr)
:=
-\log_q\norm{e_j(\pi^hY)}
=
\sum_{b\in\cA_j}\min\{h,v_\pi(b)\}
=
w_j-w_{j-h}.
\]
Thus
\[
Q_{j,h}(Y):=\pi^{w_{j-h}}f_j(\pi^hY)
\]
has Gauss norm one.
Its reduction has degree exactly $q^{j-h}$:
the factors with $v_\pi(b)<h$ reduce to nonzero constants
after normalization, while the $q^{j-h}$ remaining factors
retain their linear dependence on $Y$.

Now write $n=\ell+q^hm$, with $0\le\ell<q^h$.
The identity
\[
\eta_{n,h}
=
\sum_{j\ge h}n_jw_{j-h}
\]
shows that the reduction of $\pi^{\eta_{n,h}}G_n$ is
\[
\bigl(\overline{G_\ell(a)}\,R_m(Y)\bigr)_{a\in\cA_h},
\]
where $R_m$ is independent of $a$ and has degree exactly $m$
with nonzero leading coefficient.
The functions
\[
a\longmapsto \overline{G_\ell(a)},\qquad 0\le\ell<q^h,
\]
form a basis of all $\Fq$-valued functions on $\cA_h$:
up to nonzero scalar factors, they are the monomials of
degree less than $q$ in each of the $h$ digit coordinates.
The polynomials $R_m$, one in each degree, form a basis
of $\Fq[Y]$.
It follows that the reductions of
$\pi^{\eta_{n,h}}G_n$ form a basis of
\[
\prod_{a\in\cA_h}\Fq[Y],
\]
the reduction of the unit ball of the corresponding
$\Kinf$-Banach space.

Since $\Kinf$ is discretely valued, successive lifting
modulo powers of $\pi$ proves that these polynomials form
an orthonormal basis over $\Kinf$.
Linear independence of the reductions gives the exact
norm formula for finite sums, and completeness gives it
for convergent sums.

The same reduction argument proves the continuous assertion
over $\Kinf$.
Every continuous $\Fq$-valued function on $\cO$ factors through
$\cO/\pi^h\cO$ for some $h$, and the reductions of
$G_0,\ldots,G_{q^h-1}$ span the functions on that quotient.

Finally, pass to $\Cinf$ by isometric completed scalar extension:
\[
\mathscr E_h^{(\Kinf)}
\widehat\otimes_{\Kinf}\Cinf
\cong \mathscr E_h,
\qquad
C(\cO,\Kinf)\widehat\otimes_{\Kinf}\Cinf
\cong C(\cO,\Cinf).
\]
These are the natural isometries for the completed
non-archimedean projective tensor norm.
For the first, use the coefficient description of the finitely
many Tate algebras; for the second, use density of locally
constant functions.
An orthonormal basis identifies each space with a $c_0$-space,
and its isometric completed scalar extension carries the
same basis.
This proves all assertions over $\Cinf$.
\end{proof}

These analytic spaces give the same inductive-limit model
as before.
Restriction yields norm-nonincreasing maps
\[
B_h\longrightarrow\mathscr E_h\longrightarrow B_{h+1},
\]
since $q^{-h}<r_h$ and every level-$(h+1)$ disc lies in
its radius-$q^{-h}$ parent disc.
Consequently,
\[
\dirlim_h\mathscr E_h=\cL(\cO)
\]
with the same locally convex topology.

\begin{proposition}[Carlitz coordinates of distributions]
\label{prop:carlitz-criterion}
A sequence $(b_n)$ is of the form
\[
b_n=\mu(G_n),\qquad \mu\in\cL^*(\cO),
\]
if and only if
\begin{equation}\label{eq:carlitz-root}
\limsup_{n\to\infty}|b_n|^{1/n}\le1.
\end{equation}
The distribution is unique, and
\[
\norm{\mu}_{\mathscr E_h^*}
=
\sup_nq^{-\eta_{n,h}}|b_n|.
\]
It is a bounded measure if and only if $\sup_n|b_n|<\infty$.
\end{proposition}

\begin{proof}
Lemma~\ref{lem:weighted-carlitz-basis} shows that continuity
on $\mathscr E_h$ is equivalent to
\begin{equation}\label{eq:weighted-moment}
\sup_nq^{-\eta_{n,h}}|b_n|<\infty.
\end{equation}
Indeed, if $f=\sum_na_nG_n\in\mathscr E_h$, then
$|a_n|q^{\eta_{n,h}}\to0$, and
\[
\mu(f)=\sum_na_nb_n
\]
converges under \eqref{eq:weighted-moment}.
Testing the normalized basis gives the stated norm.
The functionals agree between levels because they agree
on polynomials, which are dense in every $\mathscr E_h$.

For fixed $h$,
\[
\eta_{n,h}
=
\frac{n}{(q-1)q^h}+O(\log(n+1)).
\]
The upper estimate follows by omitting floors.
For the lower estimate, only $O(\log n)$ floors are nonzero,
and the remaining geometric tail is bounded.
Condition \eqref{eq:weighted-moment} for every $h$ implies
\[
\limsup_n|b_n|^{1/n}
\le q^{1/((q-1)q^h)}
\]
for every $h$, and hence \eqref{eq:carlitz-root}.

Conversely, assume \eqref{eq:carlitz-root}.
Fix $h$ and choose
\[
0<\varepsilon<\frac1{(q-1)q^h}.
\]
Eventually $|b_n|\le q^{\varepsilon n}$, so
$q^{-\eta_{n,h}}|b_n|\to0$.
Thus \eqref{eq:weighted-moment} holds.
Polynomial density gives uniqueness.

The bounded-measure assertion follows from the unweighted
orthonormal basis of $C(\cO,\Cinf)$.
\end{proof}

Write
\[
G_n(X)=\sum_{k=0}^ng_{n,k}X^k.
\]
This is a triangular change of polynomial basis with
nonzero diagonal.

\begin{corollary}\label{cor:monomial-criterion}
A sequence $(m_k)$ is the ordinary moment sequence of a
distribution on $\cO$ if and only if
\[
b_n=\sum_{k=0}^ng_{n,k}m_k
\]
satisfies \eqref{eq:carlitz-root}.
The distribution is unique.
Every $\mu\in\cL^*(\cO)$ satisfies
\begin{equation}\label{eq:bounded-monomial}
|m_k(\mu)|\le\norm{\mu}_1,
\qquad
m_k(\mu)=\mu(X^k).
\end{equation}
\end{corollary}

\begin{proof}
Necessity follows by applying $\mu$ to $G_n$.
For sufficiency, construct the distribution from $(b_n)$
by Proposition~\ref{prop:carlitz-criterion}.
The inverse triangular change of basis gives the prescribed
ordinary moments.
Finally, $\norm{X^k}_1\le1$, since the level-one discs
lie in $|X|\le1$.
\end{proof}

Ordinary moments are therefore bounded even when Carlitz
moments are unbounded.
The coefficients in the change of basis are not uniformly
bounded.
Carlitz coordinates test compatibility with the specified
support set and local analyticity.

\subsection{Scalar spectral distributions}

Let $T\in\cB(X)$ be a bounded $\Cinf$-linear operator
on a $\Cinf$-Banach space.
A scalarization is a continuous linear map
\[
\Lambda:\cB(X)\longrightarrow\Cinf.
\]
For example,
\[
\Lambda(B)=\ell(Bu),\qquad u\in X,\quad\ell\in X'.
\]

\begin{definition}\label{def:admissible}
The pair $(T,\Lambda)$ is admissible if $zI-T$ is
invertible for every $z\in\Cinf\setminus\cO$ and
\[
F_{T,\Lambda}(z)=\Lambda((zI-T)^{-1})
\]
belongs to $\cH_0(\cO)$.
Its scalar spectral distribution is
\[
\mu_{T,\Lambda}=V(F_{T,\Lambda}).
\]
\end{definition}

\begin{proposition}\label{prop:spectral-moments}
For $\mu\in\cL^*(\cO)$,
\[
S_\mu(z)=\sum_{k\ge0}m_k(\mu)z^{-k-1},
\qquad |z|>1.
\]
For an admissible pair,
\[
m_k(\mu_{T,\Lambda})=\Lambda(T^k).
\]
\end{proposition}

\begin{proof}
The geometric expansion of $(z-X)^{-1}$ converges in $B_1$
for $|z|>1$.
For $|z|>\norm{T}$, the operator Neumann series gives
\[
F_{T,\Lambda}(z)
=
\sum_{k\ge0}\Lambda(T^k)z^{-k-1}.
\]
Comparison in the common convergence region proves the claim.
\end{proof}

\begin{example}[A nilpotent block]
\label{ex:jordan}
Let $T=\alpha I+J$ on a finite-dimensional space, where
$\alpha\in\cO$ and $J^r=0$.
Then
\[
(zI-T)^{-1}
=
\sum_{j=0}^{r-1}\frac{J^j}{(z-\alpha)^{j+1}},
\]
and
\[
\mu_{T,\Lambda}(f)
=
\sum_{j=0}^{r-1}\Lambda(J^j)D^{(j)}f(\alpha).
\]
The right-hand side is a locally analytic distribution whose
Stieltjes transform is the scalar resolvent.

This explains why generalized distributions, rather than
only bounded measures, arise naturally.
A scalarization may retain nilpotent information through
Hasse derivatives.
For the ordinary matrix trace, $\Tr(J^j)=0$ for $j\ge1$.
\end{example}

%======================================================================
\section{Mellin transforms and reconstruction}
\label{sec:mellin}
%======================================================================

\subsection{Principal-unit powers}

For $y\in\Zp$ and $k\ge0$, let
\[
\rbinom yk\in\Fp\subset\Cinf
\]
be the reduction of $\binom yk\in\Zp$.
For $u=1+w$ with $|w|<1$, define
\begin{equation}\label{eq:unit-power}
u^y=\sum_{k\ge0}\rbinom yk w^k.
\end{equation}
The series converges uniformly in $y$, since its $k$th term
has absolute value at most $|w|^k$.

We use the congruence
\begin{equation}\label{eq:binomial-congruence}
y\equiv y'\pmod{p^N}
\quad\Longrightarrow\quad
\rbinom yk=\rbinom{y'}k
\quad(0\le k<p^N).
\end{equation}
For nonnegative integer exponents, this follows from
$(1+X)^{p^N}=1+X^{p^N}$ in $\Fp[[X]]$.
Continuity of binomial coefficients gives the general case.
The power operation is jointly continuous for
$|w|\le\rho<1$ and $y\in\Zp$.
The usual identities extend by continuity:
\[
u^{y+y'}=u^yu^{y'},
\qquad
(uv)^y=u^yv^y,
\qquad
(u^y)^{y'}=u^{yy'}.
\]

Every group homomorphism
\[
1+\fm\longrightarrow(\Cinf,+)
\]
is zero.
Indeed, every principal unit has a principal-unit $p$th root,
whereas the additive target is killed by $p$.
This observation does not concern the Carlitz logarithm,
which linearizes a different algebraic action.

The natural ring map $\Zp\to\Cinf$ factors through $\Fp$,
whereas \eqref{eq:unit-power} generally depends on all
$p$-adic digits of $y$.
We therefore regard $y$ as a continuous $\Zp$-parameter,
not as a $\Cinf$-analytic variable.

\subsection{Quantitative Mellin calculus}

\begin{definition}\label{def:mellin}
For $\mu\in\cL^*(\cO)$, set
\[
\cM_\mu(y,t)=\mu_x((1+tx)^y),
\qquad y\in\Zp,\quad |t|<1,
\]
and
\[
\zeta_\infty^\mu(y,\lambda)
=
\cM_\mu(-y,\lambda^{-1}),
\qquad |\lambda|>1.
\]
For an admissible pair, take $\mu=\mu_{T,\Lambda}$.
\end{definition}

For $0<\rho<1$ in the value group, let $\norm{\cdot}_\rho$
denote the Gauss norm on the radius-$\rho$ Tate algebra.
We write this algebra as $\Cinf\langle t/\rho\rangle$;
a scalar of absolute value $\rho$ is understood in the
Tate coordinate.

\begin{theorem}[Quantitative Mellin calculus]
\label{thm:mellin-calculus}
For $\mu\in\cL^*(\cO)$,
\begin{equation}\label{eq:mellin-expansion}
\cM_\mu(y,t)
=
\sum_{k\ge0}\rbinom yk m_k(\mu)t^k.
\end{equation}
The transform is analytic in $t$ and continuous in $y$.
If $y\equiv y'\pmod{p^N}$, then
\begin{equation}\label{eq:sharp-continuity}
\norm{\cM_\mu(y,\cdot)-\cM_\mu(y',\cdot)}_\rho
\le
\norm{\mu}_1\rho^{p^N}.
\end{equation}
This bound is sharp as a uniform statement over distributions.

Define $\Delta_{p^N}F(y)=F(y+p^N)-F(y)$.
For $r,N\in\N_0$,
\begin{equation}\label{eq:exponent-difference}
\Delta_{p^N}^{\,r}\cM_\mu(y,t)
=
t^{rp^N}
\mu_x\bigl(x^{rp^N}(1+tx)^y\bigr).
\end{equation}
The Hasse derivatives in $t$ satisfy
\begin{equation}\label{eq:hasse-mellin}
D_t^{(r)}\cM_\mu(y,t)
=
\rbinom yr
\mu_x\bigl(x^r(1+tx)^{y-r}\bigr).
\end{equation}
\end{theorem}

\begin{proof}
Since $\norm{X^k}_1\le1$, the series
\[
(1+tX)^y=\sum_{k\ge0}\rbinom yk t^kX^k
\]
converges in $B_1$, uniformly for $y\in\Zp$ and
$|t|\le\rho<1$.
Applying $\mu$ gives \eqref{eq:mellin-expansion}.
Boundedness of ordinary moments gives analytic convergence.

By \eqref{eq:binomial-congruence}, a difference between
congruent exponents begins in degree at least $p^N$,
proving \eqref{eq:sharp-continuity}.
For $\mu=\delta_1$,
\[
\cM_{\delta_1}(p^N,t)-\cM_{\delta_1}(0,t)
=
(1+t)^{p^N}-1=t^{p^N},
\]
and $\norm{\delta_1}_1=1$, giving equality.

A shift of $y$ by $p^N$ multiplies the kernel by
$(1+tx)^{p^N}$.
Its $r$th difference therefore multiplies it by
\[
\bigl((1+tx)^{p^N}-1\bigr)^r
=
t^{rp^N}x^{rp^N},
\]
which proves \eqref{eq:exponent-difference}.

Termwise Hasse differentiation gives
\[
D_t^{(r)}\cM_\mu(y,t)
=
\sum_{k\ge r}\binom kr\rbinom yk m_k(\mu)t^{k-r}.
\]
The identity
\[
\binom yk\binom kr
=
\binom yr\binom{y-r}{k-r}
\]
holds in $\Zp$ by continuity from integer exponents.
Reduction and reindexing prove \eqref{eq:hasse-mellin}.
\end{proof}

\begin{remark}[Exponent calculus in the two characteristics]
\label{rem:padic-mellin-calculus}
For a $\Cp$-valued locally analytic distribution $\eta$ on $\Zp$,
write
\[
\cM_{\eta,p}(y,t)
=
\eta_x((1+tx)^y),
\qquad y\in\Zp,\quad |t|_p<1,
\]
using ordinary $p$-adic binomial coefficients.
On sufficiently small exponent neighborhoods,
\[
\frac{\partial^r}{\partial y^r}\cM_{\eta,p}(y,t)
=
\eta_x\bigl((1+tx)^y(\log_p(1+tx))^r\bigr).
\]
This logarithmic exponent calculus underlies the regularity
results in \cite[Section~3]{HuKimInverse}.

Our transform instead has a continuous $\Zp$-exponent and
an analytic $t$-parameter.
The exact difference identity \eqref{eq:exponent-difference}
replaces part of the role of exponent differentiation.
In characteristic zero the corresponding identity is
\[
\Delta_{p^N}^{\,r}\cM_{\eta,p}(y,t)
=
\eta_x\bigl((1+tx)^y((1+tx)^{p^N}-1)^r\bigr),
\]
but $(1+tx)^{p^N}-1$ does not reduce to $t^{p^N}x^{p^N}$.
That reduction and the exact degree cutoff are specific to
characteristic $p$.
Likewise, \eqref{eq:hasse-mellin} is the factorial-free
counterpart of the identity involving $(1/r!)\partial_t^r$.
\end{remark}

\begin{corollary}[Special values and the resolvent slice]
\label{cor:special-values}
Put
\[
P_m^\mu(\lambda)=\mu_x((\lambda+x)^m).
\]
For $m\ge0$ and $|\lambda|>1$,
\[
\zeta_\infty^\mu(-m,\lambda)
=
\lambda^{-m}P_m^\mu(\lambda),
\qquad
\left.D_t^{(m)}\cM_\mu(m,t)\right|_{t=0}=m_m(\mu).
\]
Moreover,
\begin{equation}\label{eq:zeta-resolvent}
\zeta_\infty^\mu(1,\lambda)
=
-\lambda S_\mu(-\lambda).
\end{equation}
For an admissible pair, the last expression is
\[
\lambda\Lambda((\lambda I+T)^{-1}).
\]
\end{corollary}

\begin{proof}
At a nonnegative integer exponent, the binomial series terminates.
The resolvent assertion follows from
\[
(1+x/\lambda)^{-1}
=
-\lambda(-\lambda-x)^{-1}.
\]
\end{proof}

\begin{remark}[Normalization and comparison of special values]
\label{rem:padic-zeta-normalization}
For a $\Cp$-valued distribution $\eta$ on a compact set $\Sigma$,
put
\[
\mathcal R_{\eta,p}(y,\lambda)
=
\eta_x((1+x/\lambda)^{-y}),
\qquad
|\lambda|_p>\sup_{x\in\Sigma}|x|_p.
\]
Since
\[
\langle\lambda+x\rangle
=
\langle\lambda\rangle(1+x/\lambda),
\]
the definitions in \cite[Definitions~3.2 and~3.12]{HuKimInverse}
give, on their respective domains,
\[
\mathcal R_{\mu,p}(y,\lambda)
=
y\langle\lambda\rangle^y\zeta_p^\mu(1+y,\lambda),
\qquad y\ne0,
\]
and
\[
\mathcal R_{\nu,p}(y,\lambda)
=
\langle\lambda\rangle^y\zeta_{p,E}^\nu(1+y,\lambda).
\]
The first product extends across $y=0$, with value $\mu(1)$.
The comparison involves the shift $s=1+y$ and removal of
the factor depending only on $\lambda$.

Writing $\omega_v(\lambda)=\lambda/\langle\lambda\rangle$
and $P_m^\eta(\lambda)=\eta_x((\lambda+x)^m)$, the special
values become
\[
\zeta_p^\mu(1-m,\lambda)
=
-\frac{\omega_v(\lambda)^{-m}}{m}P_m^\mu(\lambda),
\qquad m\ge1,
\]
and
\[
\zeta_{p,E}^\nu(1-m,\lambda)
=
\omega_v(\lambda)^{-m}P_m^\nu(\lambda),
\qquad m\ge0.
\]
Our relative transform instead satisfies
\[
\zeta_\infty^\mu(-m,\lambda)
=
\lambda^{-m}P_m^\mu(\lambda).
\]
The common content is evaluation on a polynomial kernel;
the other factors belong to the chosen normalizations.
Boundedness or unboundedness of a distribution does not,
by itself, force a factor $(s-1)^{-1}$.

For general distributions these are moment polynomials.
Classical Bernoulli or Euler polynomials arise only for
the corresponding Volkenborn or fermionic choices in
characteristic zero; no such identification is implicit here.
\end{remark}

\subsection{Reconstruction from a single exponent}

At $y=-1$, the reduced binomial coefficients are $(-1)^k$.
Thus uniqueness from this slice is a moment statement.
The range criterion below is the additional assertion:
it characterizes which formal coefficient sequences define
a distribution on the particular compact set $\cO$.

\begin{theorem}[Mellin reconstruction]
\label{thm:mellin-reconstruction}
Let
\[
H(t)=\sum_{k\ge0}a_kt^k\in\Cinf[[t]],
\qquad
b_n(H)=\sum_{k=0}^ng_{n,k}(-1)^ka_k.
\]
The following are equivalent:
\begin{enumerate}[(i)]
\item
There exists $\mu\in\cL^*(\cO)$ such that $H$ is the
Taylor series of $\cM_\mu(-1,t)$.
\item
\[
\limsup_{n\to\infty}|b_n(H)|^{1/n}\le1.
\]
\end{enumerate}
When these conditions hold, $\mu$ is unique,
the coefficients $a_k$ are bounded, and $H$ converges
on $|t|<1$.
The full Mellin family is
\begin{equation}\label{eq:reconstruction}
\cM_\mu(y,t)
=
\sum_{k\ge0}\rbinom yk(-1)^ka_kt^k.
\end{equation}
Furthermore,
\[
S_\mu(z)=z^{-1}H(-z^{-1}),\qquad |z|>1,
\]
and this function has an exterior continuation in $\cH_0(\cO)$.
\end{theorem}

\begin{proof}
If (i) holds, then
\[
H(t)=\sum_{k\ge0}(-1)^km_k(\mu)t^k,
\]
so $b_n(H)=\mu(G_n)$.
Proposition~\ref{prop:carlitz-criterion} gives (ii).

Conversely, (ii) constructs a unique distribution with
Carlitz moments $b_n(H)$.
The triangular change of basis gives
$m_k(\mu)=(-1)^ka_k$.
Equation \eqref{eq:bounded-monomial} implies boundedness
of $(a_k)$, and Theorem~\ref{thm:mellin-calculus} gives
\eqref{eq:reconstruction}.
Finally,
\[
(1-x/z)^{-1}=z(z-x)^{-1}
\]
gives the Stieltjes identity, and
Proposition~\ref{prop:duality} gives its exterior continuation.
\end{proof}

\begin{corollary}[A resolvent-germ criterion]
\label{cor:resolvent-germ}
A convergent germ at infinity
\[
F(z)=\sum_{k\ge0}m_kz^{-k-1}
\]
extends to an element of $\cH_0(\cO)$ if and only if
\[
\limsup_{n\to\infty}
\left|\sum_{k=0}^ng_{n,k}m_k\right|^{1/n}\le1.
\]
The extension is unique.
\end{corollary}

\begin{proof}
Apply Theorem~\ref{thm:mellin-reconstruction} to
$H(t)=\sum_{k\ge0}(-1)^km_kt^k$ and use inverse Stieltjes duality.
\end{proof}

\begin{example}[Bounded coefficients do not suffice]
\label{ex:wrong-support}
Choose $c\in\overline{\Fq}\setminus\Fq$ and consider
\[
H_c(t)=\frac1{1+ct}
=
\sum_{k\ge0}(-1)^kc^kt^k.
\]
Its coefficients have absolute value one, but it is not
the Mellin slice of a distribution on $\cO$.

Its prospective Carlitz moments are $G_n(c)$.
Every $a\in\cA_j$ has residue in $\Fq$, so $|c-a|=1$.
Hence
\[
|e_j(c)|=1,
\qquad
|G_{q^j}(c)|
=
|f_j(c)|
=
q^{(q^j-1)/(q-1)}.
\]
Therefore
\[
\lim_{j\to\infty}|G_{q^j}(c)|^{1/q^j}
=
q^{1/(q-1)}>1.
\]
The obstruction is a support condition, not failure of
convergence on the open unit disc.
\end{example}

\begin{remark}[Mahler coordinates and the support condition]
\label{rem:padic-reconstruction}
The characteristic-zero resolvent expansion also identifies
ordinary moments; see \cite[Proposition~2.23]{HuKimInverse}.
For distributions on $\Zp$, the corresponding coordinates are
the Mahler polynomials
\[
M_n(X)=\binom Xn=\sum_{k=0}^nh_{n,k}X^k.
\]
The Amice characterization of locally analytic Mahler expansions
gives the analogous dual criterion:
$(m_k)$ is the moment sequence of a locally analytic
$\Cp$-valued distribution on $\Zp$ if and only if
\[
\limsup_{n\to\infty}
\left|\sum_{k=0}^nh_{n,k}m_k\right|_p^{1/n}\le1.
\]
Bounded measures correspond to bounded Mahler moments.
This comparison comes from interpolation theory
\cite{Amice,Schikhof}, rather than a range theorem in
\cite{HuKimInverse}.

The support obstruction also persists.
Choose $c\in\Cp$ with residue outside $\Fp$ and $|c|_p=1$.
Then $|c^k|_p=1$, but
\[
\left|\binom cn\right|_p=|n!|_p^{-1},
\qquad
\lim_{n\to\infty}\left|\binom cn\right|_p^{1/n}
=
p^{1/(p-1)}>1.
\]
Thus bounded ordinary moments alone do not impose the
required support condition in either characteristic.
\end{remark}

\begin{proposition}[Entire Mellin slices]
\label{prop:entire-slice}
Let $H(t)=\sum_{k\ge0}a_kt^k$ be entire.
Then
\[
\mu_H(f)=\sum_{k\ge0}(-1)^ka_kD^{(k)}f(0)
\]
defines an element of $\cL^*(\cO)$ and
\[
\cM_{\mu_H}(-1,t)=H(t),\qquad |t|<1.
\]
The distribution vanishes on every function that is zero
on a neighborhood of $0$.
\end{proposition}

\begin{proof}
For $f\in B_n$,
\[
|D^{(k)}f(0)|\le\norm{f}_nr_n^{-k}.
\]
Entire convergence gives $|a_k|r_n^{-k}\to0$, so
\[
|\mu_H(f)|
\le
\left(\sup_k|a_k|r_n^{-k}\right)\norm{f}_n.
\]
The functionals are compatible between levels.
Since $D^{(k)}X^j|_{X=0}$ is one for $k=j$ and zero otherwise,
the moments are $(-1)^ja_j$.
The Mellin identity and the support assertion follow.
\end{proof}

This is an analytic representation statement.
It does not identify $\mu_H$ with the spectral distribution
of a previously specified operator and scalarization.

For comparison, a bounded measure has Goss's divided-power transform
\[
\mathscr P_\mu(U)=\sum_{n\ge0}\mu(G_n)U^{[n]},
\qquad
U^{[i]}U^{[j]}=\binom{i+j}{i}U^{[i+j]}.
\]
The symbols are formal and do not denote division by $n!$.
This transform converts additive convolution into
divided-power multiplication \cite{GossFormal,Kochubei}.
Our reconstruction theorem uses the same distribution
coordinates but a different kernel.

%======================================================================
\section{Digit measures and analytic fixed points}
\label{sec:digits}
%======================================================================

\subsection{The measure and its spectral realization}

Choose
\[
\beta_a\in\Cinf,\qquad a\in\Fq,\qquad
\sum_a\beta_a=1,\qquad |\beta_a|\le1.
\]
For $b=\sum_{j=0}^{N-1}a_j\pi^j\in\cA_N$, set
\[
w_{b,N}=\prod_{j=0}^{N-1}\beta_{a_j},
\qquad w_{0,0}=1.
\]
The same weight vector is used at every digit position;
we call the resulting measures stationary digit measures.

\begin{proposition}\label{prop:digit-measure}
There is a unique bounded measure $\nu_\beta$ such that
\[
\nu_\beta(b+\pi^N\cO)=w_{b,N}.
\]
Its total mass and norm are one.
It satisfies
\begin{equation}\label{eq:self-similarity}
\nu_\beta(f)
=
\sum_{a\in\Fq}\beta_a
\nu_\beta\bigl(u\mapsto f(a+\pi u)\bigr).
\end{equation}
If all $\beta_a$ are nonzero, its support is $\cO$.

Let $Q$ be multiplication by $x$ on $C(\cO,\Cinf)$.
Then $\Spec(Q)=\cO$, and
\[
\Lambda_\beta(B)=\nu_\beta(B\mathbf1)
\]
makes $(Q,\Lambda_\beta)$ an admissible pair with
spectral distribution $\nu_\beta$.
\end{proposition}

\begin{proof}
The weights satisfy
\[
\sum_{\substack{c\in\cA_{N+1}\\c\equiv b\pmod{\pi^N}}}
w_{c,N+1}
=
w_{b,N}\sum_a\beta_a
=
w_{b,N}.
\]
They define a functional on locally constant functions,
bounded by the supremum norm.
Uniform density gives its unique extension to continuous
functions.
Its value on one is one, so its norm is one.
Grouping the first digit proves \eqref{eq:self-similarity}.

If all weights are nonzero, every residue class has nonzero
mass, and every nonempty open set contains such a class.
This gives full support.

For $z\notin\cO$, multiplication by $(z-x)^{-1}$ is the
bounded inverse of $zI-Q$.
For $z\in\cO$, every function in its range vanishes at $z$,
so $zI-Q$ is not surjective.
Hence $\Spec(Q)=\cO$.
Finally,
\[
\Lambda_\beta((zI-Q)^{-1})
=
\nu_\beta((z-x)^{-1})
=
S_{\nu_\beta}(z).
\]
Proposition~\ref{prop:duality} proves admissibility.
\end{proof}

The weights are field-valued, not real probabilities.
Equal weights of total mass one are impossible in
characteristic $p$.
Nonzero bounded weights exist: choose
\[
\gamma\in\overline{\Fq}\setminus\{0,-1\},
\qquad
\beta_0=1+\gamma,\qquad
\beta_a=\gamma\quad(a\ne0).
\]
Their sum is $1+q\gamma=1$.
This gives a full-support scalar spectral distribution.

\begin{remark}[Khrennikov's position operators and the choice of measure]
\label{rem:padic-position}
The position-operator realization is the counterpart of
\cite{Khrennikov1991} and \cite[Section~5]{HuKimInverse}.
In both characteristics, multiplication by the coordinate
on $C(\Sigma)$ has spectrum $\Sigma$, and a chosen bounded
measure $\eta$ gives the scalarization
$\Lambda_\eta(B)=\eta(B\mathbf1)$.
The operator alone does not specify a canonical scalar
spectral distribution.

Boundedness must be distinguished from Haar normalization.
On $\Zp$, the $\Cp$-valued Haar distribution assigns mass
$p^{-N}$ to a residue class modulo $p^N$ and is unbounded.
Its Volkenborn extension to locally analytic functions is not
translation invariant \cite{Schikhof}.
A normalized translation-invariant functional on all locally
analytic functions would already contradict the identity
$(x+1)-x=1$.

In characteristic $p$, a normalized translation-invariant
field-valued distribution on locally constant functions
would give
\[
1=q\,\eta(\pi\cO)=0.
\]
The digit measures avoid this obstruction through unequal
weights. They are not normalized Haar measures.
\end{remark}

\subsection{A contraction principle}

For fixed $y\in\Zp$, define
\[
(\mathcal T_{\beta,y}F)(t)
=
\sum_{a\in\Fq}\beta_a(1+at)^y
F\left(\frac{\pi t}{1+at}\right).
\]

\begin{theorem}[Digit fixed-point theorem]
\label{thm:digit-fixed-point}
For every $y\in\Zp$, the function $\cM_{\nu_\beta}(y,t)$
is the unique analytic function on $|t|<1$ satisfying
\[
F=\mathcal T_{\beta,y}F,\qquad F(0)=1.
\]
On $\Cinf\langle t/\rho\rangle$, $0<\rho<1$,
the operator preserves the condition $F(0)=1$ and satisfies
\[
\norm{\mathcal T_{\beta,y}F-\mathcal T_{\beta,y}G}_\rho
\le
|\pi|\norm{F-G}_\rho
\]
whenever $F(0)=G(0)$.

If $F_{N,y}=\mathcal T_{\beta,y}^{\,N}1$, then
\begin{equation}\label{eq:digit-finite-sum}
F_{N,y}(t)
=
\sum_{b\in\cA_N}w_{b,N}(1+bt)^y,
\end{equation}
and
\begin{equation}\label{eq:digit-error}
\sup_{y\in\Zp}
\norm{F_{N,y}-\cM_{\nu_\beta}(y,\cdot)}_\rho
\le
\rho|\pi|^N.
\end{equation}
\end{theorem}

\begin{proof}
On the radius-$\rho$ disc, $1+at$ and its inverse have norm one,
and
\[
\norm{\frac{\pi t}{1+at}}_\rho\le|\pi|\rho,
\qquad
\norm{(1+at)^y}_\rho=1.
\]
If $H(t)=\sum_{k\ge1}h_kt^k$, then
\[
\norm{H\left(\frac{\pi t}{1+at}\right)}_\rho
\le
\sup_{k\ge1}|h_k|(|\pi|\rho)^k
\le
|\pi|\norm{H}_\rho.
\]
The ultrametric inequality and $|\beta_a|\le1$ give the
contraction estimate.
The constant term is preserved because $\sum_a\beta_a=1$.

The factorization
\[
1+t(a+\pi u)
=
(1+at)\left(1+\frac{\pi t}{1+at}u\right)
\]
and \eqref{eq:self-similarity} show that the Mellin transform
is a normalized fixed point.
The affine subspace $F(0)=1$ is complete, so contraction
gives uniqueness on every smaller closed disc.

Formula \eqref{eq:digit-finite-sum} follows inductively from
\[
(1+at)^y
\left(1+\frac{\pi bt}{1+at}\right)^y
=
(1+(a+\pi b)t)^y.
\]
Since $\nu_\beta$ has norm and mass one,
\[
\norm{\cM_{\nu_\beta}(y,\cdot)-1}_\rho\le\rho.
\]
Applying the contraction $N$ times gives
\eqref{eq:digit-error}, uniformly in $y$.
\end{proof}

\subsection{Recovery of the digit weights}

Write
\[
m_k=\nu_\beta(x^k),
\qquad
c_j=\sum_{a\in\Fq}\beta_aa^j,
\qquad m_0=c_0=1,
\]
with $a^0=1$.

\begin{corollary}[Finite-moment identification]
\label{cor:digit-identification}
For $k\ge1$,
\begin{equation}\label{eq:digit-moment}
(1-\pi^k)m_k
=
\sum_{j=0}^{k-1}\binom kj\pi^jc_{k-j}m_j.
\end{equation}
Within the class of normalized stationary digit measures
defined above, the $q-1$ moments
\[
m_1,\ldots,m_{q-1}
\]
determine all weights $\beta_a$ and hence the measure.
Equivalently, the coefficients of $t,\ldots,t^{q-1}$ in
$\cM_{\nu_\beta}(-1,t)$ determine the measure within this class;
the constant coefficient is already one.
\end{corollary}

\begin{proof}
Self-similarity gives
\[
m_k=\sum_{j=0}^k\binom kj\pi^jc_{k-j}m_j.
\]
The last term is $\pi^km_k$, proving \eqref{eq:digit-moment}.
Since $m_0=1$,
\[
c_k
=
(1-\pi^k)m_k
-
\sum_{j=1}^{k-1}\binom kj\pi^jc_{k-j}m_j
\]
recursively determines $c_1,\ldots,c_{q-1}$.
The equations
\[
c_k=\sum_{a\in\Fq}\beta_aa^k,\qquad 0\le k<q,
\]
have an invertible Vandermonde matrix.
The Mellin coefficients at $y=-1$ are $(-1)^km_k$.
\end{proof}

\begin{remark}[The $p$-adic fermionic measure as a digit model]
\label{rem:padic-digit-fixed-point}
The contraction argument also applies in characteristic zero.
For odd $p$, use digits $a=0,\ldots,p-1$, weights
$\beta_a=(-1)^a$, and replace $\pi$ by $p$.
Since
\[
\sum_{a=0}^{p-1}(-1)^a=1,
\qquad
(-1)^{\sum_ja_jp^j}=\prod_j(-1)^{a_j},
\]
the resulting stationary digit measure is the fermionic
measure $\mu_{-1}$ of \cite[Definition~2.3]{HuKimInverse}.
Its transform satisfies
\[
\cM_{\mu_{-1},p}(y,t)
=
\sum_{a=0}^{p-1}(-1)^a(1+at)^y
\cM_{\mu_{-1},p}
\left(y,\frac{pt}{1+at}\right),
\]
and is the unique normalized analytic fixed point, with
contraction factor $p^{-1}$.
For
\[
R_{N,y}(t)=\sum_{b=0}^{p^N-1}(-1)^b(1+bt)^y,
\]
the same proof gives
\[
\sup_{y\in\Zp}
\norm{R_{N,y}-\cM_{\mu_{-1},p}(y,\cdot)}_\rho
\le\rho p^{-N}.
\]
Related bounded-measure resolvent approximations appear in
\cite[Section~3.2]{HuKimInverse}.

By contrast, the Haar weights $1/p$ have absolute value $p$.
The analogous estimate gives only
\[
|p|_p\max_a|1/p|_p=1,
\]
not a strict contraction.
The bounded-digit argument therefore does not replace the
Volkenborn estimates in \cite[Theorem~3.21]{HuKimInverse}.
Finite-moment recovery of stationary digit weights also
carries over, using the Vandermonde matrix on $0,\ldots,p-1$.
\end{remark}

%======================================================================
\section{Frobenius-normalized spectral determinants}
\label{sec:fredholm}
%======================================================================

A scalar distribution obtained from traces need not retain
all determinant data in characteristic $p$.
Over a characteristic-zero field, a finite-dimensional
operator satisfies the formal identity
\[
\det(I-ZT)
=
\exp\left(-\sum_{m\ge1}\frac{\Tr(T^m)}mZ^m\right).
\]
Thus power traces determine the determinant.
Equivalently, the logarithmic derivative and $F(0)=1$
determine a formal series uniquely.

In characteristic $p$, differentiation annihilates every
series in $Z^p$.
Even a small-coefficient normalization does not remove this
ambiguity:
for $0<|c|<1$, the distinct series
\[
1,\qquad 1-c^pZ^p
\]
have the same logarithmic derivative.
For a fixed finite dimension less than $p$, however, Newton's
recursion still determines the characteristic polynomial,
since the required integers remain invertible.

The characteristic-zero comparison concerns ordinary traces,
not arbitrary scalarizations.
Even there, $\Lambda(T^m)$ need not determine an ordinary
operator determinant.
The normalized log-gamma functionals of
\cite[Section~4]{HuKimInverse} are defined by exponent
differentiation of scalar spectral zeta functions.
Their moment expansions do not assert that every scalar
resolvent determines an ordinary Fredholm determinant.

We instead prescribe coefficient-Frobenius data independently
of scalar traces.
The small-coefficient condition gives uniqueness only after
those data have been supplied.
The formal Frobenius theory does not require eigenvalues to
belong to $\cO$.
The support hypothesis is imposed separately when connecting
determinants with the preceding distribution theory.

\subsection{What scalar traces retain}

\begin{proposition}[Trace ambiguity in characteristic $p$]
\label{prop:trace-ambiguity}
For a finite-dimensional operator $T$,
\[
\det(I-ZT)=\sum_r(-1)^r\Tr(\wedge^rT)Z^r,
\]
and
\[
\frac{\det(I-ZT)'}{\det(I-ZT)}
=
-\sum_{m\ge1}\Tr(T^m)Z^{m-1}.
\]
Scalar power traces do not determine the determinant in general.
More precisely, for $F,G\in1+Z\Cinf[[Z]]$,
\[
F'/F=G'/G
\quad\Longleftrightarrow\quad
F/G\in1+Z^p\Cinf[[Z^p]].
\]
\end{proposition}

\begin{proof}
The exterior-power formula expresses the coefficients as sums
of principal minors.
Jacobi's adjugate formula and the formal inverse of $I-ZT$
give the logarithmic derivative.

Take $T_1=I_p$ and $T_2=0_p$.
All their power traces, including degree zero, are zero, but
\[
\det(I-ZT_1)=1-Z^p,\qquad
\det(I-ZT_2)=1.
\]
Their trace-scalarized analytic spectral distributions also
coincide, since
\[
\Tr(f(T_1))=pf(1)=0,\qquad
\Tr(f(T_2))=pf(0)=0.
\]
Finally, with $R=F/G$, equality of logarithmic derivatives
is equivalent to $R'=0$, which is equivalent to all
nonconstant terms having degree divisible by $p$.
\end{proof}

\subsection{Normalized coefficient-Frobenius equations}
\label{sec:frobenius-resolvent}

Let
\[
\mathcal U=
\left\{
1+\sum_{r\ge1}a_rZ^r:\ a_r\in\fm
\right\}.
\]
This is an abelian group under multiplication.
The condition is coefficientwise; no common bound less than
one is assumed.
Every coefficient of a product or inverse involves only
finitely many additions and multiplications, so these operations
preserve the condition.

Define
\begin{equation}\label{eq:coefficient-frobenius}
\sigma\left(\sum_{r\ge0}a_rZ^r\right)
=
\sum_{r\ge0}a_r^qZ^r.
\end{equation}
Since $\Cinf$ is perfect, $\sigma$ is a ring automorphism and
restricts to an automorphism of $\mathcal U$.
It fixes $Z$ and differs from ordinary exponentiation:
\[
F(Z)^q=\sigma(F)(Z^q).
\]
For Gauss norms,
\begin{equation}\label{eq:frobenius-gauss}
\norm{\sigma(F)}_\rho
=
\norm{F}_{\rho^{1/q}}^q
\end{equation}
whenever the right-hand side is defined.

\begin{definition}[Frobenius defect]
\label{def:frobenius-defect}
For $F\in\mathcal U$, put
\[
\Def(F)=F/\sigma(F).
\]
For prescribed $\Phi\in\mathcal U$, a normalized solution is
an $F\in\mathcal U$ satisfying
\begin{equation}\label{eq:frobenius-equation}
F=\Phi\,\sigma(F).
\end{equation}
\end{definition}

Without the small-coefficient condition, a solution can be
multiplied by a nontrivial series fixed by $\sigma$.

\begin{lemma}[Small Artin--Schreier solution]
\label{lem:artin-schreier}
For $b\in\fm$, the equation $x^q-x=b$ has a unique solution
in $\fm$, namely
\[
x=-\sum_{j\ge0}b^{q^j}.
\]
It satisfies $|x|=|b|$.
\end{lemma}

\begin{proof}
The series converges and telescopes under Frobenius.
Two solutions differ by a root of $X^q-X$, hence by an
element of $\Fq$.
The only such element in $\fm$ is zero.
For a nonzero small solution, $|x^q-x|=|x|$.
The zero case is immediate.
\end{proof}

\begin{theorem}[Normalized Frobenius resolution]
\label{thm:frobenius-resolution}
The map $\Def:\mathcal U\to\mathcal U$ is a group automorphism.
For
\[
\Phi=1+\sum_{r\ge1}\phi_rZ^r,
\qquad
\mathfrak F(\Phi)=\Def^{-1}(\Phi)
=
1+\sum_{r\ge1}d_rZ^r,
\]
the coefficients are uniquely determined by
\begin{equation}\label{eq:resolution-recursion}
d_0=1,\qquad
d_r^q-d_r
=
-\sum_{i=1}^r\phi_id_{r-i}^q,
\end{equation}
using the solution in $\fm$ at each step.

Moreover,
\begin{equation}\label{eq:frobenius-infinite-product}
\mathfrak F(\Phi)
=
\prod_{j=0}^\infty\sigma^j(\Phi).
\end{equation}
The product converges coefficientwise in the valuation topology
and in the Gauss norm on every closed disc $|Z|\le\rho<1$.

Both $\Def$ and its inverse preserve the $Z$-adic filtration.
For $F,G\in\mathcal U$,
\begin{equation}\label{eq:defect-isometry}
\norm{\Def(F)-\Def(G)}_\rho
=
\norm{F-G}_\rho,
\qquad 0<\rho<1.
\end{equation}
\end{theorem}

\begin{proof}
Multiplicativity follows from commutativity and multiplicativity
of $\sigma$.
Coefficient comparison gives \eqref{eq:resolution-recursion}.
Its right-hand side depends only on earlier coefficients
and lies in $\fm$.
Lemma~\ref{lem:artin-schreier} gives a unique $d_r\in\fm$
inductively, proving bijectivity.

Fix $\rho<1$ and put
\[
\varepsilon_j(\rho)
=
\norm{\sigma^j(\Phi)-1}_\rho
=
\sup_{r\ge1}|\phi_r|^{q^j}\rho^r.
\]
This decreases to zero.
For $r>R$, the terms are bounded by $\rho^{R+1}$;
the finitely many terms with $r\le R$ tend to zero.

Each factor has norm one.
For
\[
P_M=\prod_{j=0}^{M-1}\sigma^j(\Phi),
\]
the ultrametric product estimate gives
\[
\norm{P_{M'}-P_M}_\rho
\le
\sup_{j\ge M}\varepsilon_j(\rho)
=
\varepsilon_M(\rho),
\qquad M'>M.
\]
Thus the products converge on every smaller closed disc
and coefficientwise.

The limit belongs to $\mathcal U$.
For fixed $r$, let
\[
a_r=\max_{1\le i\le r}|\phi_i|<1.
\]
Every contribution to a positive-degree coefficient of degree
at most $r$ contains a nonconstant factor of norm at most
$a_r$, while all other factors have norm at most one.
The relevant coefficients of the finite products and their
limits have norm at most $a_r$.

The finite products satisfy $P_M=\Phi\,\sigma(P_{M-1})$.
Passing to the coefficientwise limit gives
\eqref{eq:frobenius-equation}, and uniqueness proves
\eqref{eq:frobenius-infinite-product}.

If
\[
F=1+a_sZ^s+\text{higher-degree terms},\qquad a_s\ne0,
\]
then
\[
\Def(F)
=
1+(a_s-a_s^q)Z^s+\text{higher-degree terms}.
\]
Since $|a_s|<1$, the displayed coefficient is nonzero.
Applying this to quotients proves preservation of the filtration.

Finally, write $F=1+H$.
Coefficientwise, $|h_r-h_r^q|=|h_r|$, so
\[
\norm{H-\sigma(H)}_\rho=\norm{H}_\rho.
\]
The series $\sigma(F)$ and its inverse have norm one.
Therefore
\[
\norm{\Def(F)-1}_\rho
=
\norm{\frac{H-\sigma(H)}{\sigma(F)}}_\rho
=
\norm{F-1}_\rho.
\]
Apply this identity to $F/G$.
Multiplication or division by $G$ and $\Def(G)$ preserves
the norm, since these units and their inverses have norm one.
This proves \eqref{eq:defect-isometry}.
\end{proof}

\begin{remark}[Topology and analytic scope]
The isometry concerns $\Def$, not $\sigma$.
The product \eqref{eq:frobenius-infinite-product} generally
does not converge in the $Z$-adic topology:
its coefficients converge in absolute value rather than
eventually stabilizing.

Elements of $\mathcal U$ are analytic on the open unit disc,
but need not be entire.
For example,
\[
1+\frac{cZ}{1-Z}\in\mathcal U,\qquad 0<|c|<1,
\]
is not entire.
\end{remark}

\begin{remark}[Why coefficient Frobenius is different over $\Cp$]
\label{rem:padic-frobenius}
Theorem~\ref{thm:frobenius-resolution} uses two properties:
$a\mapsto a^q$ is a ring automorphism, and its iterates
send every small coefficient to zero.
Over $\Cp$, this power map is not additive.

A genuine unital Frobenius lift on a suitable
characteristic-zero coefficient ring does not automatically
restore the theorem.
Such a lift fixes $p$; if it fixes $Z$, then
\[
\sigma(1+pZ)=1+pZ.
\]
Thus a nontrivial small-coefficient series lies in the kernel
of $F\mapsto F/\sigma(F)$.

The scalar equation $x^q-x=b$ still has a unique small solution
over $\Cp$ for $|b|_p<1$, by contraction.
What is special here is compatibility with additive Frobenius,
the explicit Artin--Schreier series, and the multiplicative
resolution.
The polynomial-defect classification below describes additional
Frobenius structure on an eigenvalue multiset; it is not a
reformulation of a $p$-adic Stirling expansion.
\end{remark}

\subsection{Diagonal Fredholm determinants and polynomial defects}

We specify the determinant convention needed in this paper.
A bounded linear operator is called completely continuous
here if it is a norm limit of finite-rank operators.
This terminology is used in non-archimedean Fredholm theory;
see \cite{Serre}.
It is not an assertion that the image of every bounded set
is topologically relatively compact.

For a finite-rank operator $B$, define $\det(I-ZB)$ by
restricting to any finite-dimensional invariant subspace
containing its image.
The result is independent of that subspace, since the induced
operator on an added quotient is zero.

\begin{definition}[Diagonal Fredholm determinant]
\label{def:diagonal-fredholm}
Let $T_\lambda$ be the diagonal operator on $c_0(\Cinf)$ with
entries $\lambda_n\to0$.
Let $T_\lambda^{(M)}$ agree with $T_\lambda$ on the first
$M+1$ standard basis vectors and be zero on the remaining ones.
Its Fredholm determinant is
\[
\Delta_\lambda(Z)
=
\det(I-ZT_\lambda)
:=
\lim_{M\to\infty}\det(I-ZT_\lambda^{(M)})
=
\prod_{n\ge0}(1-\lambda_nZ).
\]
The limit is taken in the Tate norm on every closed disc.
Its existence and zero structure are proved in the next lemma.
\end{definition}

This is the diagonal specialization of the usual
non-archimedean Fredholm determinant.
All determinant arguments below use this explicit construction,
so no general spectral theorem for arbitrary Banach spaces
is required.

\begin{lemma}[Convergent diagonal products]
\label{lem:diagonal-products}
Let $\lambda_n\to0$ in $\Cinf$.
Then
\[
\Delta_\lambda(Z)=\prod_{n\ge0}(1-\lambda_nZ)
\]
converges in the Tate norm on every closed disc and defines
an entire function.
Its zeros are precisely the reciprocals of the nonzero
$\lambda_n$, with multiplicity equal to the number of
occurrences of that eigenvalue.
Zero entries do not affect the product.

If, in addition, $\lambda_n\in\fm$ for every $n$, then
$\Delta_\lambda\in\mathcal U$.
\end{lemma}

\begin{proof}
Fix a radius $R$ in the value group.
For sufficiently large $M$,
\[
\norm{\prod_{n=M+1}^{M'}(1-\lambda_nZ)-1}_R
\le
R\sup_{n>M}|\lambda_n|\longrightarrow0.
\]
The finite initial factors have bounded product norm on
the fixed disc, so the full products are Cauchy.
The construction is unchanged by reordering, since only
finitely many factors fail any prescribed small-tail bound.

A sufficiently far tail differs from one by norm less than
one on the disc and is therefore an analytic unit.
The zeros are those of a finite initial product, with the
usual multiplicities.
Every fixed nonzero eigenvalue occurs finitely often.

If all $|\lambda_n|<1$, their convergence to zero gives
an $a<1$ with $|\lambda_n|\le a$ for every $n$.
The degree-$r$ coefficients of the finite products have norm
at most $a^r$.
The same holds in the limit, proving membership in $\mathcal U$.
\end{proof}

\begin{theorem}[Polynomial-defect products]
\label{thm:frobenius-geometric}
Let $\Phi\in\mathcal U\cap\Cinf[Z]$.
If $\Phi\ne1$, it factors as
\[
\Phi(Z)=\prod_{i=1}^s(1-c_iZ),
\qquad c_i\in\fm\setminus\{0\},
\]
with repetitions allowed.
Its normalized solution is entire and equals
\begin{equation}\label{eq:geometric-product}
\mathfrak F(\Phi)(Z)
=
\prod_{i=1}^s\prod_{n\ge0}(1-c_i^{q^n}Z).
\end{equation}
For $\Phi=1$, the solution is $1$.

For a single seed $c\ne0$, write
\[
\mathfrak F(1-cZ)=\sum_{r\ge0}d_rZ^r.
\]
Then
\[
d_0=1,\qquad d_r^q-d_r=c\,d_{r-1}^q,
\]
and
\begin{equation}\label{eq:geometric-valuations}
|d_r|=|c|^{(q^r-1)/(q-1)},
\qquad
v_\pi(d_r)=\frac{q^r-1}{q-1}v_\pi(c).
\end{equation}
\end{theorem}

\begin{proof}
Since $\Cinf$ is algebraically closed and $\Phi(0)=1$,
the polynomial factors into terms $1-c_iZ$ with $c_i\ne0$.
It has no root in $|Z|\le1$, because $|\Phi(Z)-1|<1$ there.
Hence $|c_i|<1$.

For one seed, Lemma~\ref{lem:diagonal-products} shows that
\[
P_c(Z)=\prod_{n\ge0}(1-c^{q^n}Z)
\]
is entire and lies in $\mathcal U$.
Frobenius shifts the product:
\[
\sigma(P_c)(Z)=\prod_{n\ge1}(1-c^{q^n}Z).
\]
Thus $P_c=(1-cZ)\sigma(P_c)$, and uniqueness gives
$P_c=\mathfrak F(1-cZ)$.
Multiplicativity proves \eqref{eq:geometric-product}.

Coefficient comparison gives the recursion.
Lemma~\ref{lem:artin-schreier} gives
\[
|d_r|=|c|\,|d_{r-1}|^q,
\]
which proves \eqref{eq:geometric-valuations}.

There is also a direct check:
\[
d_r=(-1)^r
\sum_{0\le n_1<\cdots<n_r}
c^{q^{n_1}+\cdots+q^{n_r}}.
\]
The term of smallest valuation is unique and corresponds
to $\{n_1,\ldots,n_r\}=\{0,\ldots,r-1\}$.
Its exponent is $1+q+\cdots+q^{r-1}$.
The ultrametric inequality therefore gives exactly the
same valuation formula.
\end{proof}

For several seeds, cancellations may occur among coefficients.
The exact valuation formula is asserted only for one seed.

\begin{corollary}[Classification by polynomial defect]
\label{cor:frobenius-determination}
Let $(\lambda_n)\subset\fm$ tend to zero.
Then $\Def(\Delta_\lambda)$ is polynomial if and only if
there is a finite multiset
$\{c_1,\ldots,c_s\}\subset\fm\setminus\{0\}$ such that
\[
\{\lambda_n:\lambda_n\ne0\}
=
\bigsqcup_{i=1}^s\{c_i^{q^j}:j\ge0\}
\]
as multisets.
The empty multiset is allowed.
In this case,
\[
\Def(\Delta_\lambda)=\prod_{i=1}^s(1-c_iZ),
\]
and this polynomial uniquely determines the seed multiset.
\end{corollary}

\begin{proof}
If the defect is polynomial, factor it as in
Theorem~\ref{thm:frobenius-geometric}.
The diagonal determinant and the corresponding orbit product
are normalized solutions of the same equation, so they agree.
Lemma~\ref{lem:diagonal-products} identifies their nonzero
eigenvalue multisets from reciprocal zeros, including
multiplicities.

Conversely, the multiset identity gives the orbit product,
whose defect is the displayed polynomial.
Regrouping is valid by uniform tail convergence on every
closed disc.
Polynomial factorization gives uniqueness of the seeds.
\end{proof}

The polynomial condition is substantive.
For example,
\[
\Delta(Z)=1-cZ,\qquad 0<|c|<1,
\]
has defect
\[
\frac{1-cZ}{1-c^qZ},
\]
which is not polynomial.
Its one-element spectrum is not a complete forward Frobenius orbit.
The classification concerns nonzero diagonal eigenvalue
multisets, not all operator structures or zero-eigenspace
dimensions.

\subsection{The Mellin transform of a logarithmic derivative}

To connect determinant products with distributions on $\cO$,
the eigenvalues must belong to $\cO$.
Smallness in $\Cinf$ alone does not imply this.

\begin{proposition}[Weighted diagonal spectral identity]
\label{prop:diagonal-mellin}
Suppose
\[
\lambda_n\in\cO\cap\fm,\qquad \lambda_n\to0.
\]
Let $T_\lambda$ be the corresponding diagonal operator
on $c_0(\Cinf)$, and define
\[
\nu_\lambda=\sum_{n\ge0}\lambda_n\delta_{\lambda_n},
\qquad
\Lambda_\lambda(B)=\sum_{n\ge0}\lambda_nB_{nn}.
\]
Then $\nu_\lambda$ is a bounded measure,
$\Lambda_\lambda$ is continuous, and
$(T_\lambda,\Lambda_\lambda)$ is admissible with
spectral distribution $\nu_\lambda$.
Moreover,
\[
m_k(\nu_\lambda)=\sum_{n\ge0}\lambda_n^{k+1},
\]
and
\begin{equation}\label{eq:diagonal-mellin}
\frac{\Delta_\lambda'(Z)}{\Delta_\lambda(Z)}
=
-\cM_{\nu_\lambda}(-1,-Z),
\qquad |Z|<1.
\end{equation}
\end{proposition}

\begin{proof}
The norms of $\lambda_n\delta_{\lambda_n}$ tend to zero,
so the measure series converges in the Banach dual of
$C(\cO,\Cinf)$.
Since $|B_{nn}|\le\norm{B}$, the scalarization is also
well-defined and continuous.

For $z\notin\cO$, the diagonal inverse exists and
\[
\Lambda_\lambda((zI-T_\lambda)^{-1})
=
\sum_{n\ge0}\frac{\lambda_n}{z-\lambda_n}
=
S_{\nu_\lambda}(z).
\]
Inverse Stieltjes duality proves admissibility.
The moment formula is immediate.

For finite products,
\[
\frac{\left(\prod_{n=0}^M(1-\lambda_nZ)\right)'}
{\prod_{n=0}^M(1-\lambda_nZ)}
=
-\sum_{n=0}^M\frac{\lambda_n}{1-\lambda_nZ}.
\]
On each closed disc of radius less than one, the products
and their inverses converge.
Differentiation is continuous in the Tate norm, so passage
to the limit proves \eqref{eq:diagonal-mellin}.
\end{proof}

The weights $\lambda_n$ are essential.
The unweighted sum $\sum_n\delta_{\lambda_n}$ is not
defined by norm convergence, and evaluation on the constant
function would involve the divergent series $\sum_n1$.

\begin{remark}[Resolvents, log-gamma functionals, and Fredholm products]
\label{rem:padic-logarithmic-comparison}
For sufficiently large $|\lambda|_p$, write
\[
L_{\mu,p}(\lambda)=\log\Gamma_p^\mu(\lambda),
\qquad
L_{\nu,p,E}(\lambda)=\log\Gamma_{p,E}^\nu(\lambda).
\]
The normalized integral representations in
\cite[Section~4]{HuKimInverse} are
\[
L_{\mu,p}(\lambda)
=
\mu_x\bigl((\lambda+x)(\log_p(\lambda+x)-1)\bigr)
\]
and
\[
L_{\nu,p,E}(\lambda)
=
-\nu_x\bigl((\lambda+x)\log_p(\lambda+x)\bigr).
\]
Differentiating twice in $\lambda$ gives
\[
L_{\mu,p}''(\lambda)
=
\mu_x((\lambda+x)^{-1})
=
-S_\mu(-\lambda),
\]
and
\[
L_{\nu,p,E}''(\lambda)
=
-\nu_x((\lambda+x)^{-1})
=
S_\nu(-\lambda).
\]
These differentiations are justified locally in a fixed
analytic neighborhood of the support.

Thus the Stieltjes kernel appears after two derivatives
of the normalized $p$-adic log-gamma functionals.
In the present setting it appears instead as the logarithmic
derivative of a diagonal Fredholm product:
\[
\frac{\Delta_\lambda'(Z)}{\Delta_\lambda(Z)}
=
-\cM_{\nu_\lambda}(-1,-Z).
\]
The weighted measure here is specified by the diagonal
eigenvalues.

The $p$-adic Stirling expansion uses the ordinary logarithm
and denominators such as $n(n+1)$, so it cannot be transferred
term by term to characteristic $p$.
Here the Mellin slice determines the logarithmic derivative,
whereas independently prescribed Frobenius data determine
the normalized determinant.
No general identification of these determinant constructions
is asserted.
\end{remark}

\subsection{The shifted Carlitz operator}
\label{sec:shifted-carlitz}

Consider the Carlitz--Bargmann--Fock-type space studied in
\cite[Section~1.7]{Kochubei}; see also \cite{KochubeiOscillator}.
Using the local parameter $\pi$, define
\[
\mathscr H=
\left\{
u(X)=\sum_{n\ge0}a_nu_n(X):
a_n\in\Cinf,\ a_n\to0
\right\},
\qquad
u_n(X)=\frac{X^{q^n}}{\CD_n},
\]
with
\[
\norm{u}_{\mathscr H}=\sup_n|a_n|.
\]
Since
\[
|\CD_n|=q^{-(q^n-1)/(q-1)},
\]
these series define analytic functions for
$|X|<q^{-1/(q-1)}$.
The basis $u_n$ differs from the additive Carlitz
polynomials $e_n$ used earlier.

Let
\[
\tau u=u^q,\qquad
\nabla_\pi u(X)=u(\pi X)-\pi u(X),
\qquad
d=\tau^{-1}\nabla_\pi.
\]
The inverse $\tau^{-1}$ is applied only to the range
of $\nabla_\pi$, which lies in the range of $\tau$.
The coefficient formulas are
\[
d(cu_n)=c^{1/q}u_{n-1}\quad(n\ge1),
\qquad d(cu_0)=0,
\]
and
\[
\tau(cu_n)=c^q[n+1]_\pi u_{n+1}.
\]
Thus
\[
\tau(cu)=c^q\tau(u),
\qquad
d(cu)=c^{1/q}d(u).
\]
Both operators are additive and continuous.
More precisely,
\[
\norm{\tau(u)-\tau(v)}_{\mathscr H}
=
q^{-1}\norm{u-v}_{\mathscr H}^{\,q},
\]
and
\[
\norm{d(u)-d(v)}_{\mathscr H}
\le
\norm{u-v}_{\mathscr H}^{\,1/q}.
\]
The operator $d$ is not Lipschitz on any neighborhood of zero:
for $c\ne0$,
\[
\frac{\norm{d(cu_1)}_{\mathscr H}}{\norm{cu_1}_{\mathscr H}}
=
|c|^{1/q-1}\longrightarrow\infty
\quad\text{as }c\to0.
\]

\begin{proposition}\label{prop:carlitz-operator}
The composite $N=\tau d=\nabla_\pi$ is bounded and
$\Cinf$-linear, with
\[
Nu_n=[n]_\pi u_n,\qquad \norm{N}=q^{-1},
\]
and
\[
\Spec(N)=\{[n]_\pi:n\ge0\}\cup\{-\pi\}.
\]
Its shift $A=N+\pi I$ has diagonal entries
\[
\alpha_n=\pi^{q^n},
\]
spectrum
\[
\Spec(A)=\{0\}\cup\{\alpha_n:n\ge0\},
\]
and is completely continuous.
\end{proposition}

\begin{proof}
The semilinear twists cancel:
\[
N\left(\sum_na_nu_n\right)
=
\sum_na_n[n]_\pi u_n.
\]
For $n\ge1$, $|[n]_\pi|=q^{-1}$, while $[0]_\pi=0$.
This proves linearity and the norm formula.

Outside the closure of the diagonal entries,
$(z-[n]_\pi)^{-1}$ are uniformly bounded and define the
diagonal resolvent.
Every diagonal entry is an eigenvalue.
Since $[n]_\pi\to-\pi$ and
\[
\norm{(N+\pi I)u_n}_{\mathscr H}=|\pi|^{q^n}\to0,
\]
there is no bounded inverse at $-\pi$.

The shifted entries are $\alpha_n$.
Let $A^{(M)}$ agree with $A$ on $u_0,\ldots,u_M$
and be zero on the remaining basis vectors.
Then
\[
\norm{A-A^{(M)}}=q^{-q^{M+1}}\to0.
\]
Thus $A$ is completely continuous.
The same diagonal argument gives its spectrum.
\end{proof}

The relation $\alpha_{n+1}=\alpha_n^q$ supplies the
Frobenius defect directly from the operator, independently
of trace calculations.

\begin{theorem}[The shifted Carlitz determinant]
\label{thm:frobenius-determinant}
The determinants of the finite-rank truncations $A^{(M)}$
converge in the Tate norm on every closed disc to
\begin{equation}\label{eq:fredholm-product}
\Delta_A(Z)
=
\det(I-ZA)
=
\mathfrak F(1-\pi Z)
=
\prod_{n\ge0}(1-\pi^{q^n}Z).
\end{equation}
It is the unique $F\in\mathcal U$ satisfying
\begin{equation}\label{eq:frobenius-determinant}
F=(1-\pi Z)\sigma(F).
\end{equation}
If $\Delta_A(Z)=\sum_{r\ge0}d_rZ^r$, then
\[
d_0=1,\qquad d_r^q-d_r=\pi d_{r-1}^q,
\]
and
\[
v_\pi(d_r)=\frac{q^r-1}{q-1}.
\]
Its Newton polygon has successive slopes
$1,q,q^2,\ldots$, each of horizontal length one.
Its zeros are precisely $\pi^{-q^n}$, $n\ge0$, and are simple.
\end{theorem}

\begin{proof}
The finite-rank determinants are the finite products
in \eqref{eq:fredholm-product}.
Lemma~\ref{lem:diagonal-products} gives entire convergence.
The equation, uniqueness, recursion, and valuations follow
from Theorem~\ref{thm:frobenius-geometric} with $c=\pi$.

The consecutive valuation differences are
\[
v_\pi(d_r)-v_\pi(d_{r-1})=q^{r-1}.
\]
They increase strictly, so every coefficient point is a
vertex of the lower Newton polygon.
The $\pi^{q^n}$ have distinct absolute values, and the
zero assertion follows from Lemma~\ref{lem:diagonal-products}.
\end{proof}

\begin{remark}[The Anderson--Thakur function]
\label{rem:anderson-thakur}
Set $\theta=\pi^{-1}$ and choose $\xi\in\Cinf$ with
$\xi^{q-1}=-\theta$.
The Anderson--Thakur function has the standard product
representation
\[
\omega_\theta(Z)
=
\xi\prod_{n\ge0}
\left(1-\frac{Z}{\theta^{q^n}}\right)^{-1};
\]
see \cite{AndersonThakur,AnglesPellarin}.
Initially the product is defined for $|Z|<|\theta|$,
and it has a meromorphic continuation.
Since $\omega_\theta(0)=\xi$, direct comparison gives
\[
\Delta_A(Z)
=
\frac{\omega_\theta(0)}{\omega_\theta(Z)}.
\]
The ratio is independent of $\xi$ and extends to the
entire function \eqref{eq:fredholm-product}.

The standard equation
\[
\sigma(\omega_\theta)=(Z-\theta)\omega_\theta
\]
is consistent with \eqref{eq:frobenius-determinant},
because $\xi^q=-\theta\xi$.
Thus the operator supplies a Fredholm realization of a
normalized reciprocal of an established function.
The product identity itself follows directly from the
classical formula and is not claimed as a new
Anderson--Thakur identity.
Nor does it identify the local interpolation system with
the global Carlitz module.
\end{remark}

\begin{remark}[Different roles of the spectral examples]
\label{rem:padic-carlitz-example}
The position operator illustrates inverse Stieltjes
reconstruction for a prescribed measure on a compact spectral
set, as in \cite[Section~5]{HuKimInverse}.
That construction does not imply that the position operator
has an ordinary Fredholm product.

The shifted Carlitz operator has a different role:
its nonzero eigenvalues tend to zero, its finite-rank
truncations converge in operator norm, and their determinants
converge to an entire product.
The relation $\alpha_{n+1}=\alpha_n^q$ supplies extra
Frobenius data.

The paper therefore contains two complementary models:
a position operator with non-discrete compact spectrum,
and a completely continuous operator with discrete nonzero
spectrum accumulating at zero.
The Anderson--Thakur identification belongs to the second
model and is not an automatic consequence of inverse
Stieltjes duality.
\end{remark}

Put
\[
s_m=\sum_{n\ge0}\alpha_n^m\quad(m\ge1),
\qquad
\nu_A=\sum_{n\ge0}\alpha_n\delta_{\alpha_n},
\]
and
\[
\Lambda_A(B)=\sum_{n\ge0}\alpha_nB_{nn},
\]
where matrix entries are taken in the basis $(u_n)$.

\begin{corollary}[Weighted Carlitz spectral transform]
\label{cor:weighted-mellin}
The pair $(A,\Lambda_A)$ is admissible and has spectral
distribution $\nu_A$.
Its moments satisfy
\[
m_k(\nu_A)=s_{k+1},
\]
and
\begin{equation}\label{eq:weighted-mellin}
\frac{\Delta_A'(Z)}{\Delta_A(Z)}
=
-\cM_{\nu_A}(-1,-Z),
\qquad |Z|<1.
\end{equation}
Moreover,
\begin{equation}\label{eq:trace-artin-schreier}
s_m^q-s_m=-\pi^m.
\end{equation}
\end{corollary}

\begin{proof}
All $\alpha_n$ lie in $\cO\cap\fm$ and tend to zero,
so Proposition~\ref{prop:diagonal-mellin} applies.
Continuity and additivity of Frobenius give
\[
s_m^q
=
\sum_{n\ge0}\pi^{mq^{n+1}}
=
s_m-\pi^m.
\]
\end{proof}

\subsection{Analytic representation and arithmetic scope}

Determinant reconstruction uses two inputs:
a prescribed Frobenius defect and the small-coefficient
normalization.
In the Carlitz example, the diagonal relation supplies
the defect independently.
Changing scalar traces into Carlitz moments or
divided-power coordinates would not supply this missing
information.

Belonging to a Mellin range is also different from being
the transform of a natural spectral distribution.
Every polynomial-defect solution is entire, so
Proposition~\ref{prop:entire-slice} represents it as the
Mellin slice of a distribution supported at zero in the
sense stated there.
That distribution is generally different from the weighted
spectral measure whose transform gives the logarithmic derivative.

For general $\Phi\in\mathcal U$, the normalized solution is
only known to be analytic on the open unit disc.
The Carlitz criterion then remains an additional condition,
as the following example shows.

\begin{example}[Small coefficients do not imply the Mellin range condition]
\label{ex:small-outside-range}
Choose
\[
c\in\overline{\Fq}\setminus\Fq,
\qquad 0<|\varepsilon|<1,
\]
and put
\[
F(t)
=
1-\frac{\varepsilon ct}{1+ct}
=
1+\sum_{k\ge1}\varepsilon(-c)^kt^k.
\]
Then $F\in\mathcal U$, with the formal variable renamed $t$.
Its prospective ordinary moments as a $y=-1$ Mellin slice are
\[
m_0=1,\qquad m_k=\varepsilon c^k\quad(k\ge1).
\]
For $n\ge1$, $G_n(0)=0$, since every nonconstant $G_n$
contains a factor $f_j$ vanishing at zero.
Hence the prospective Carlitz moments satisfy
\[
b_n
=
(1-\varepsilon)G_n(0)+\varepsilon G_n(c)
=
\varepsilon G_n(c).
\]
By Example~\ref{ex:wrong-support},
\[
\lim_{j\to\infty}|b_{q^j}|^{1/q^j}
=
q^{1/(q-1)}>1.
\]
Thus $F$ is not the Mellin slice of a distribution on $\cO$.

Nevertheless,
\[
\Phi=\Def(F)\in\mathcal U,
\]
and $F=\mathfrak F(\Phi)$ is its unique normalized solution.
Therefore neither membership in $\mathcal U$ nor normalized
solvability of a prescribed Frobenius equation, by itself,
implies the Mellin range condition.
\end{example}

An arithmetic application requires the defect to arise from
a natural operator or module identity.
For a Drinfeld module, the original Frobenius satisfies
\[
\tau(cx)=c^q\tau(x)
\]
and is generally not $\Cinf$-linear.
A comparison with Taelman's nuclear determinants
\cite{Taelman} must specify the underlying module,
coefficient ring, determinant theory, and scalarization.
Determinant laws provide another framework for retaining
characteristic-polynomial data over general rings
\cite{Chenevier}, but are not identified with our construction.

The resulting question is not merely whether a formal
Frobenius equation exists.
It is whether its data and normalization are supplied by
the arithmetic structure, and whether the scalar transform
has the intended spectral meaning.
The Anderson--Thakur comparison above gives one precise
connection without asserting a general identification with
Goss or Taelman special values.

%======================================================================
\section{Conclusion}
\label{sec:conclusion}
%======================================================================

The resolvent approach separates the construction of a scalar
spectral distribution from the choice of a zeta kernel.
On $\cO$, inverse Stieltjes duality and weighted Carlitz
coordinates give an explicit criterion for reconstructing a
distribution from a single Mellin exponent.
The content of this criterion is its range and support
condition, not merely the formal recovery of ordinary moments.
For stationary digit measures, self-similarity provides a
normalized fixed-point characterization, uniform approximation,
and identification from finitely many nonconstant moments
within that class.

The comparison with the $p$-adic theory separates common
non-archimedean mechanisms from positive-characteristic ones.
Inverse Stieltjes duality, moment expansions, and bounded-digit
contraction persist in characteristic zero.
Polynomial special values have the same kernel-level origin,
although exponent shifts and normalizations differ.
By contrast, the exact Frobenius exponent-difference formula
and the loss of multiplicity information in scalar traces
depend on characteristic $p$.

Determinants therefore require additional structure.
A prescribed coefficient-Frobenius defect has a unique
small-coefficient solution, obtained by Artin--Schreier
recursion or a convergent product.
Polynomial defects give entire functions and characterize
diagonal spectra formed from finitely many forward Frobenius
orbits, with multiplicities.
The weighted spectral Mellin transform gives the logarithmic
derivative, while the Frobenius equation determines the
normalized determinant.
This differs from the normalized $p$-adic log-gamma
construction, despite their common relation to the
Stieltjes kernel.

The shifted Carlitz operator makes the extra structure explicit.
Its diagonal relation supplies the defect, and comparison
with the standard Anderson--Thakur product identifies its
Fredholm determinant.
Further arithmetic applications require similarly natural
operator or module identities, together with a compatible
scalarization and normalization.

%======================================================================

\end{document}